\documentclass[12pt,a4paper]{article} 
\usepackage[colorlinks]{hyperref}
\usepackage{amsmath}
\usepackage{amsthm}
\usepackage{amsfonts}
\usepackage{amssymb}
\usepackage{mathabx}
\usepackage{bussproofs}
\usepackage{lineno}
\theoremstyle{plain} 
\newtheorem{thm}{Theorem}[section]

\usepackage{tikz}

\usepackage{titlesec}

\titleformat*{\section}{\normalfont\fontfamily{phv}\fontsize{16}{19}\bfseries}
\titleformat*{\subsection}{\normalfont\fontfamily{put}\fontsize{12}{15}\bfseries}
\titleformat*{\subsubsection}{\normalfont\fontfamily{put}\fontsize{14}{17}\selectfont}

\newtheorem{lem}[thm]{Lemma}

\newtheorem{cor}[thm]{Corollary}
\theoremstyle{definition}
\newtheorem{dfn}[thm]{Definition}

\newtheorem{rem}[thm]{Remark}

\def\NP{\mathrm{NP}}
\def\P{\mathrm{P}}
\def\PSPACE{\mathrm{PSPACE}}
\def\coNP{\mathrm{coNP}}

\begin{document}
\title{Proof Complexity of Substructural Logics} 

\author{
Raheleh Jalali \footnote{Supported by the grant 19-05497S of GA \v{C}R, RVO: 67985840, and the grant 639.073.807 of the Netherlands Organisation for Scientific Research. 
}
\vspace{10pt}\\
Utrecht University\\
The Netherlands
\vspace{3pt}\\
\texttt{rahele.jalali@gmail.com}\\
}

\date{\today}
\maketitle

\begin{abstract}  
In this paper, we investigate the proof complexity of a wide range of substructural systems. For any proof system $\mathbf{P}$ at least as strong as Full Lambek calculus, $\mathbf{FL}$, and polynomially simulated by the extended Frege system for some super-intuitionistic logic of infinite branching, we present an exponential lower bound on the proof lengths. More precisely, we will provide a sequence of $\mathbf{P}$-provable formulas $\{A_n\}_{n=1}^{\infty}$ such that the length of the shortest $\mathbf{P}$-proof for $A_n$ is exponential in the length of $A_n$. The lower bound also extends to the number of proof-lines (proof-lengths) in any Frege system (extended Frege system) for a logic between $\mathsf{FL}$ and any super-intuitionistic logic of infinite branching. We will also prove a similar result for the proof systems and logics extending Visser's basic propositional calculus $\mathbf{BPC}$ and its logic $\mathsf{BPC}$, respectively. Finally, in the classical substructural setting, we will establish an exponential lower bound on the number of proof-lines in any proof system polynomially simulated by the cut-free version of $\mathbf{CFL_{ew}}$. \\
\textbf{Keywords:} proof complexity, substructural logics, sub-intuitionistic logics. \textsc{MSC}: 03F20, 03B47, 03B20
\end{abstract}

\section{Introduction}
Propositional proof complexity, as a new independent field, was established predominantly to address the fundamental unsolved problems in computational complexity. Starting steps in this systematic study were taken by Cook and Reckhow. In their seminal paper \cite{Cook}, they defined a propositional proof system, PPS, as a polynomial-time computable function whose range is the set of all classical propositional tautologies. Then, they defined a polynomially bounded proof system as a PPS having a short proof for any tautology, i.e., a proof whose length is polynomially bounded by the length of the tautology itself. They proved that the existence of a polynomially-bounded proof system for the classical logic is equivalent to $\NP=\coNP$. Accordingly, if for any PPS there are super-polynomial lower bounds on the lengths of proofs, as a result $\NP$ will be different from $\coNP$ and consequently, $\P$ will be different from $\NP$. Since these are considered to be major open problems in computational complexity, providing super-polynomial lower bounds for all PPS's gained momentum in the field of proof complexity of classical proof systems. Thus far, exponential lower bounds on proof lengths have been established in many different propositional proof systems, including resolution \cite{Haken}, cutting planes \cite{Cutting2}, and bounded-depth Frege systems \cite{Frege2}. For more on the lengths of proofs, see \cite{KrajicekProof}.\\

Aside from the extensive study of some well-known classical proof systems, recently there have been some investigations into the complexity of proofs in non-classical logics on account of their various applications, their power in expressibility and their essential role in computer science. Therefore, it is important to fully understand the inherent complexity of proofs in non-classical logics, considering specially the impact that lower bounds on lengths of proofs will have on the performance of the proof search algorithms. Moreover, from the computational complexity perspective, the study of complexity of proofs in non-classical logics is associated with another major computational complexity problem, namely the $\NP$ vs. $\PSPACE$ problem. Various results have been acheived in this area, for instance exponential lower bounds for the intuitionistic and modal logics \cite{Hrubes}, and for modal and intuitionistic Frege and extended Frege systems \cite{Jerabek}. 
A comprehensive overview of results concerning proof complexity of non-classical logics can be found in \cite{Olaf}.\\

In the realm of non-classical logics, substructural ones are logics originally defined by the systems where some or all of the usual structural rules are absent. These logics include relevant logics, linear logic, fuzzy logics, and many-valued logics. However, the field is more ambitious than any limited investigation of possible effects of the structural rules. The purpose of the study of substructural logics is to uniformly investigate the non-classical logics that originated from different motivations. Complexity-theoretically, several substructural logics are $\PSPACE$-complete, for instance the multiplicative-additive fragment of linear logic, $\mathbf{MALL}$ \cite{MALL}, and full Lambek calculus, $\mathbf{FL}$ \cite{FL}. Check also the $\PSPACE$-hardness for a wide range of substructural logics and $\PSPACE$-completeness for a class of extensions of $\mathbf{FL}$ in \cite{Terui}. Some complexity results about the decision problem of some fragments of Visser's basic propositional logic, $\mathsf{BPC}$, and formal propositional logic, $\mathsf{FPL}$ are also studied in \cite{Rybakov}.\\

In this paper, we will study the proof complexity of proof systems for substructural logics and super-basic logic, and hence a wide-range class of proof systems. More precisely, we will start with an arbitrary proof system $\mathbf{P}$ at least as strong as $\mathbf{FL}$ (or $\mathbf{BPC}$) and polynomially simulated by an extended Frege system for some super-intuitionistic logic $\mathsf{L}$ of infinite branching, denoted by $\mathsf{L}-\mathbf{EF}$. For such a $\mathbf{P}$, we will provide a sequence of hard $\mathbf{P}$-tautologies, namely a sequence of $\mathbf{P}$-provable formulas $\{A_n\}_{n=1}^{\infty}$ with length polynomial in $n$ such that their shortest $\mathbf{P}$-proofs are exponentially long in $n$. Our method is using a sequence of intuitionistic tautologies for which we know there exists an exponential lower bound on the length of their proofs in any $\mathsf{L}-\mathbf{EF}$, where $\mathsf{L}$ has infinite branching. Since these formulas are not necessarily provable in $\mathbf{P}$, the essential step is their modification so that they become provable in $\mathbf{FL}$ (or $\mathbf{BPC}$) and hence in $\mathbf{P}$, while they remain hard for $\mathsf{L}-\mathbf{EF}$. Finally, since $\mathsf{L}-\mathbf{EF}$ is shown to be polynomially as strong as $\mathbf{P}$, the length of any $\mathbf{P}$-proofs of the $\mathbf{P}$-tautologies must be exponential in $n$.
Furthermore, using the same $\mathbf{FL}$-tautologies, one can infer an exponential lower bound also for proof systems polynomially simulated by $\mathbf{CFL_{ew}^-}$, where the superscript ``$-$" means the sequent calculus does not have the cut rule.\\

The outline of the paper is as follows. In Section \ref{SectionPreliminaries}, we start with extensive preliminaries and recall substructural and super-basic logics. Section \ref{SectionFregeExtendedFrege} discusses proof systems and Frege and extended Frege systems for substructural and super-basic logics. Section \ref{SectionDescent} presents a method to convert tautologies in classical logic to tautologies in substructural logics. We then recall the form of hard intuitionistic tautologies in subsection \ref{SubsectionDigression}. Theorem \ref{HardTautologiesFLe} introduces tautologies in substructural logics, whose form is similar to the form of hard intuitionistic tautologies. A translation of these formulas will also be provable in basic logic. Theorem \ref{MainTheorem} and Theorem \ref{ThmFLBPC-EF} in section \ref{SectionLowerBound} state our main results about the existence of an exponential lower bound on the lengths of proofs in a wide range of proof systems for substructural and super-basic logics, including Frege and extended Frege systems for a broad range of substructural and super-basic logics. Moreover, Corollary \ref{CorKolli} is the main concrete application of the paper. Section \ref{SectionSequentCalculus} deals with Gentzen-style sequent calculi for substructural logics and basic logic and provides an exponential lower bound on the number of proof-lines in these systems. Moreover, it discusses how to provide an exponential lower bound on the number of proof-lines in cut-free Gentzen-style sequent calculi for the classical counterparts of some basic substructural logics.
%states that the lower bound in proof systems for some single-conclusion substructural logics can be used to provide an exponential lower bound on the number of proof-lines in cut-free Gentzen-style sequent calculi for their classical counterparts. 

\section{Preliminaries} \label{SectionPreliminaries}
In this section we provide some background and also some new notions needed in the future sections. Throughout the paper we mainly work with substructural logics and we follow \cite{Ono} as the canonical source for the study of the theory of such logics. Nevertheless, to make the paper as self-contained as possible, we include all necessary background information.

\subsection{Substructural logics}
Consider the propositional language $\{\wedge, \vee, *, \top, \bot, 1, 0, /, \setminus, \to \}$. The logical connective $*$ is called fusion and the connectives $\setminus$ and $/$ are called left division and right division, respectively. Throughout the paper, small Roman letters, $p, q, \ldots$ are reserved for propositional variables, small Greek letters $\phi$, $\psi$, $\ldots$, and capital Roman letters $A$, $B$, $\ldots$ denote formulas, the letters $\Gamma , \Sigma, \Pi, \Delta, \Lambda, \Phi$ denote (possibly empty) finite sequences of formulas, separated by commas, and the letters $\Upsilon, \Xi, \Omega$ are reserved for (possibly empty) finite multisets of formulas, unless specified otherwise. All of these letters are used possibly with sub- and/or superscripts. Throughout the paper, we will use the convention that for a logic or a proof system $S$, $\vdash_{S} \phi$ and $S \vdash \phi$ are used interchangeably. If $S$ is a sequent calculus, we sometimes write $S \vdash \Gamma \Rightarrow \Delta$ for $\vdash_{S} \Gamma \Rightarrow \Delta$.\\

Consider the following set of rules that we will use to define basic substructural logics. They are expressed in terms of sequents of the form $\Gamma \Rightarrow \Delta$, where $\Gamma$ is called the antecedent of the sequent and $\Delta$ its succedent. Sequents above the line in each rule are called premises and the sequent below the line is called the conclusion. All the rules are presented in the form of schemes. Therefore, an instance of a rule is obtained by substituting formulas for lower case letters (also called formula variables) and finite (possibly empty) sequences of formulas for upper case letters.

\begin{flushleft}
 \textbf{Initial sequents:}
\end{flushleft}
\begin{center}
$\phi\Rightarrow \phi\; $ \;\;\;\;\; $\Gamma \Rightarrow \Delta, \top, \Lambda $ \;\;\;\;\; $ \Gamma, \bot, \Sigma \Rightarrow \Delta$ \;\;\;\;\;
$\Rightarrow 1$ \;\;\;\;\;\;\;\;\; $0 \Rightarrow$
\end{center}

\begin{flushleft}
 \textbf{Structural rules:}
\end{flushleft}

\begin{itemize}
\item[]
\begin{flushleft}
{Weakening rules:}
\end{flushleft}

\begin{center}
 \begin{tabular}{c c } 
 \AxiomC{$\Gamma, \Sigma \Rightarrow \Delta$}
 \RightLabel{$(L w)$} 
 \UnaryInfC{$\Gamma, \phi, \Sigma \Rightarrow \Delta$}
 \DisplayProof
 &
 \AxiomC{$\Gamma \Rightarrow \Delta, \Lambda$}
 \RightLabel{$(R w)$} 
 \UnaryInfC{$\Gamma \Rightarrow \Delta, \phi, \Lambda $}
 \DisplayProof
\end{tabular}
\end{center}

\begin{flushleft}
{Contraction rules:}
\end{flushleft}

\begin{center}
 \begin{tabular}{c c } 
 \AxiomC{$\Gamma, \phi, \phi, \Sigma \Rightarrow \Delta$}
 \RightLabel{$(L c)$} 
 \UnaryInfC{$\Gamma, \phi, \Sigma \Rightarrow \Delta$}
 \DisplayProof
 &
 \AxiomC{$\Gamma \Rightarrow \Delta, \phi, \phi, \Lambda$}
 \RightLabel{$(R c)$} 
 \UnaryInfC{$\Gamma \Rightarrow \Delta, \phi, \Lambda $}
 \DisplayProof
\end{tabular}
\end{center}

\begin{flushleft}
{Exchange rules:}
\end{flushleft}

\begin{center}
 \begin{tabular}{c c} 
 \AxiomC{$\Gamma, \phi, \psi, \Sigma \Rightarrow \Delta$}
  \RightLabel{$(L e)$} 
 \UnaryInfC{$ \Gamma, \psi, \phi, \Sigma \Rightarrow \Delta$}
 \DisplayProof
 &
  \AxiomC{$\Gamma \Rightarrow \Delta, \phi, \psi, \Lambda$}
  \RightLabel{$(R e)$} 
 \UnaryInfC{$ \Gamma \Rightarrow \Delta, \psi, \phi, \Lambda$}
 \DisplayProof
\end{tabular}
\end{center}
\end{itemize}

\begin{flushleft}
 \textbf{The cut rule:}
\end{flushleft}

\begin{center}
 \begin{tabular}{c}
 \AxiomC{$\Gamma \Rightarrow \phi, \Lambda$}
 \AxiomC{$\Sigma, \phi, \Pi \Rightarrow \Delta$} 
   \RightLabel{$(cut)$}
 \BinaryInfC{$ \Sigma, \Gamma, \Pi \Rightarrow \Delta, \Lambda$}
 \DisplayProof
\end{tabular}
\end{center}

\begin{flushleft}
 \textbf{The logical rules:}
\end{flushleft}

\begin{center}
 \begin{tabular}{c c } 
 \AxiomC{$\Gamma, \Sigma \Rightarrow \Delta$}
 \RightLabel{$(1 w)$} 
 \UnaryInfC{$\Gamma, 1, \Sigma \Rightarrow \Delta$}
 \DisplayProof
 &
 \AxiomC{$\Gamma \Rightarrow \Delta, \Lambda$}
 \RightLabel{$(0 w)$} 
 \UnaryInfC{$\Gamma \Rightarrow \Delta, 0, \Lambda $}
 \DisplayProof
\end{tabular}
\end{center}

\begin{center}
 \begin{tabular}{c c} 
 \AxiomC{$\Gamma, \phi, \Sigma \Rightarrow \Delta$}
 \RightLabel{$(L \wedge_1)$} 
 \UnaryInfC{$\Gamma, \phi \wedge \psi, \Sigma \Rightarrow \Delta$}
 \DisplayProof
 &
  \AxiomC{$\Gamma, \psi, \Sigma \Rightarrow \Delta$}
 \RightLabel{$(L \wedge_2)$} 
 \UnaryInfC{$\Gamma, \phi \wedge \psi, \Sigma \Rightarrow \Delta$}
 \DisplayProof
\end{tabular}
\end{center}

\begin{center}
 \begin{tabular}{c} 
 \AxiomC{$\Gamma \Rightarrow \Delta, \phi, \Lambda$}
 \AxiomC{$\Gamma \Rightarrow \Delta, \psi, \Lambda$}
 \RightLabel{$(R \wedge)$} 
 \BinaryInfC{$\Gamma \Rightarrow \Delta, \phi \wedge \psi, \Lambda$}
 \DisplayProof
\end{tabular}
\end{center}

\begin{center}
 \begin{tabular}{c}
 \AxiomC{$\Gamma, \phi, \Sigma \Rightarrow \Delta$}
 \AxiomC{$\Gamma, \psi, \Sigma \Rightarrow \Delta$}
 \RightLabel{$(L \vee)$} 
 \BinaryInfC{$\Gamma, \phi \vee \psi, \Sigma \Rightarrow \Delta$}
 \DisplayProof
\end{tabular}
\end{center}

\begin{center}
 \begin{tabular}{c c}
\AxiomC{$\Gamma \Rightarrow \Delta, \phi, \Lambda$}
 \RightLabel{$(R \vee_1)$} 
 \UnaryInfC{$\Gamma \Rightarrow \Delta, \phi \vee \psi, \Lambda$}
 \DisplayProof
 &
 \AxiomC{$\Gamma \Rightarrow \Delta, \psi, \Lambda$}
 \RightLabel{$(R \vee_2)$} 
 \UnaryInfC{$\Gamma \Rightarrow \Delta, \phi \vee \psi, \Lambda$}
 \DisplayProof
\end{tabular}
\end{center}

\begin{center}
 \begin{tabular}{c c } 
 \AxiomC{$\Gamma, \phi, \psi, \Sigma \Rightarrow \Delta$}
 \RightLabel{$(L*)$} 
 \UnaryInfC{$\Gamma, \phi * \psi, \Sigma \Rightarrow \Delta$}
 \DisplayProof
 &
 \AxiomC{$\Gamma \Rightarrow \Delta, \phi, \Lambda$}
 \AxiomC{$\Sigma \Rightarrow \Delta, \psi, \Lambda$}
 \RightLabel{$(R*)$} 
 \BinaryInfC{$\Gamma, \Sigma \Rightarrow \Delta, \phi * \psi, \Lambda$}
 \DisplayProof
\end{tabular}
\end{center}

\begin{flushleft}
 \textbf{The non-commutative implications rules:}
\end{flushleft}

\begin{center}
 \begin{tabular}{c c } 
 \AxiomC{$\Gamma \Rightarrow \phi$}
 \AxiomC{$\Pi, \psi, \Sigma \Rightarrow \delta$}
 \RightLabel{$(L/)$} 
 \BinaryInfC{$\Pi, \psi / \phi , \Gamma, \Sigma \Rightarrow \delta$}
 \DisplayProof
 &
 \AxiomC{$\Gamma, \phi \Rightarrow \psi$}
 \RightLabel{$(R/)$} 
 \UnaryInfC{$\Gamma \Rightarrow \psi / \phi$}
 \DisplayProof
\end{tabular}
\end{center}

\begin{center}
 \begin{tabular}{c c } 
 \AxiomC{$\Gamma \Rightarrow \phi$}
 \AxiomC{$\Pi, \psi, \Sigma \Rightarrow \delta$}
 \RightLabel{$(L \setminus)$} 
 \BinaryInfC{$\Pi, \Gamma, \phi \setminus \psi, \Sigma \Rightarrow \delta$}
 \DisplayProof
 &
 \AxiomC{$\phi, \Gamma \Rightarrow \psi$}
 \RightLabel{$(R \setminus)$} 
 \UnaryInfC{$\Gamma \Rightarrow \phi \setminus \psi$}
 \DisplayProof
\end{tabular}
\end{center}

\begin{flushleft}
 \textbf{The commutative implication rules:}
\end{flushleft}

\begin{center}
 \begin{tabular}{c c}
 \AxiomC{$\Gamma \Rightarrow \phi, \Lambda$}
 \AxiomC{$\Pi, \psi, \Sigma \Rightarrow \Delta$}
 \RightLabel{$(L \to)$} 
 \BinaryInfC{$\Pi, \Gamma, \phi \to \psi, \Sigma \Rightarrow \Delta, \Lambda$}
 \DisplayProof
 &
 \AxiomC{$\phi, \Gamma \Rightarrow \psi, \Delta$}
 \RightLabel{$(R \to)$} 
 \UnaryInfC{$\Gamma \Rightarrow \phi \to \psi, \Delta$}
 \DisplayProof
\end{tabular}
\end{center}

Using these rules, we define two families of sequent-style systems in the following; single-conclusion and multi-conclusion.\\

\textit{Single-conclusion}. By a \emph{single-conclusion sequent} we mean a sequent whose succedent has at most one formula. Otherwise, we call it \emph{multi-conclusion}. By a single-conclusion version of any of the aforementioned rules, we mean one of its instances, where both the premisses and the conclusion sequents are single-conclusion. Therefore, $\Delta$ and $\Lambda$, as schematic variables, must be replaced by the empty sequence or a single formula so that all the sequents remain single-conclusion. For instance, in the rule $(Rw)$ or in the initial sequent for $\top$, both $\Delta$ and $\Lambda$ must be empty. Notice that the rules $(Rc)$ and $(Re)$ do not have a single-conclusion instance. We will use the convention that $*$ binds more tightly than $\setminus$ and $/$.
The \emph{interpretation} of any single-conclusion sequent $\Gamma \Rightarrow \phi$ is defined as $I(\Gamma \Rightarrow \phi)= \bigast \Gamma \setminus \phi$ and for the sequent $(\Gamma \Rightarrow)$ as $I(\Gamma \Rightarrow )=\bigast \Gamma \setminus 0$, where by $\bigast \Gamma$ for $\Gamma=\gamma_1, \ldots, \gamma_n$, we mean $\gamma_1 * \ldots * \gamma_n$, and for $\Gamma= \emptyset$, we mean $\bigast \Gamma = 1$. \\
Set $\mathcal{L^*}= \{\wedge, \vee, *, \setminus, /, 1, 0\}$. Let $(e)$, $(c)$, $(i)$, $(o)$, and $(w)=(i + o)$ stand for exchange, contraction, left-weakening, right-weakening and weakening, respectively. 
For any $S \subseteq \{e, i, o, c\}$, define the sequent-style system $\mathbf{FL_S}$ over the language $\mathcal{L^*}$ as the system consisting of the single-conclusion version of the previous rules except for: the commutative implication rules, the structural rules out of the set $S$, and the initial sequents for $\bot$ and $\top$. Define $\mathbf{FL}_{\bot}$ over the language $\mathcal{L^*} \cup \{\bot\}$ as $\mathbf{FL}$ with the single-conclusion version of the initial sequent for $\bot$. Figure \ref{figure}, adapted from \cite{Ono}, shows the relationship between these sequent calculi. Moreover, define the system \textit{weak Lambek}, denoted by $\mathbf{WL}$, over the language $\{1, \bot, \wedge, \vee, *, \setminus\}$, similar to $\mathbf{FL_{\bot}}$, excluding the following rules: $(L/), (R/),$ and $(L \setminus)$.
Some other useful calculi are introduced in Table \ref{table:pekh}. To define these systems, for a sequent calculus $\mathbf{S}$ and a set of sequents $\mathcal{I}$, let the notation $\mathbf{S}+ \mathcal{I}$ be the sequent calculus obtained from adding the elements of $\mathcal{I}$ as initial sequents to $\mathbf{S}$. 
By the notation $\phi \Leftrightarrow \psi$ we mean both $\phi \Rightarrow \psi$ and $\psi \Rightarrow \phi$. The formula $\phi^n$ is defined inductively. $\phi^1$ is $\phi$ and by $\phi^{n+1}$, we mean $\phi * \phi^n$.

\begin{table}[ht]
\caption{Some sequent calculi with their definitions.} % title of Table
\centering % used for centering table
\vline
\begin{tabular}{c c} % centered columns (4 columns) 
\hline\hline %inserts double horizontal lines
Sequent calculus & Definition \\ %[0.5ex] % inserts table
%heading
\hline % inserts single horizontal line
$\mathbf{RL}$ & $\mathbf{FL}+ (0 \Leftrightarrow 1)$ \\ % inserting body of the table
\hline
$\mathbf{CyFL}$ & $\mathbf{FL}+ (\phi \setminus 0 \Leftrightarrow 0 / \phi)$ \\
\hline
$\mathbf{DFL}$ & $\mathbf{FL}+ (\phi \wedge (\psi \vee \theta) \Leftrightarrow (\phi \wedge \psi) \vee (\phi \wedge \theta))$  \\
\hline
$\mathbf{P_nFL}$ &  $\mathbf{FL}+ (\phi^n \Leftrightarrow \phi^{n+1})$ \\
\hline
$\mathbf{psBL}$ & $\mathbf{FL_w}+ \{(\phi \wedge \psi \Leftrightarrow \phi * (\phi \setminus \psi)), (\phi \wedge \psi \Leftrightarrow (\psi / \phi) * \phi) \}$\\
\hline
$\mathbf{HA}$ & $\mathbf{FL_w}+ (\phi \Leftrightarrow \phi^{2})$\\
\hline
$\mathbf{DRL}$ & $\mathbf{RL} + (\phi \wedge (\psi \vee \theta) \Leftrightarrow (\phi \wedge \psi) \vee (\phi \wedge \theta))$\\
\hline
$\mathbf{IRL}$ & $\mathbf{RL} + (\phi \Rightarrow 1)$\\
\hline
$\mathbf{CRL}$ & $\mathbf{RL} + (\phi * \psi \Leftrightarrow \psi * \phi)$\\
\hline
$\mathbf{GBH}$ & $\mathbf{RL} + \{(\phi \wedge \psi \Leftrightarrow \phi * (\phi \setminus \psi)), (\phi \wedge \psi \Leftrightarrow (\psi / \phi) * \phi) \}$\\
\hline
$\mathbf{Br}$ & $\mathbf{RL} + (\phi \wedge \psi \Leftrightarrow \phi * \psi)$\\%[1ex] % [1ex] adds vertical space
\hline %inserts single line
\end{tabular}\vline
\label{table:pekh} % is used to refer this table in the text
\end{table}

\textit{Multi-conclusion}.
In the absence of the exchange rules, there are several possible ways to define the multi-conclusion rules for fusion and implications and the systems are in some respects more difficult than the commutative case. 
In this paper, we only consider the commutative case and hence we will use the language $\{\wedge, \vee, *, \to, 0, 1\}$, assuming only one implication. The interpretation of any sequent $\Gamma \Rightarrow \Delta$ is defined as $I(\Gamma \Rightarrow \Delta)=\bigast \Gamma \to \neg (\bigast \neg \Delta)$, where $\neg \phi$ is an abbreviation for $\phi \to 0$ and for $\Delta= \delta_1, \ldots, \delta_m$ by $\neg \Delta$, we mean the sequence $\neg \delta_1, \ldots, \neg \delta_m$.\\
Let $S \subseteq \{e, i, o, c\}$ such that $e \in S$. By $\mathbf{CFL_{S}}$, we mean the system consisting of the multi-conclusion version of the previous rules except for: the structural rules out of the set $S$, the non-commutative implication rules, and the initial sequents for $\bot$ and $\top$. By $\mathbf{CFL_S}^-$, we mean $\mathbf{CFL_S}$ without the cut rule.\\
For a sequent calculus $\mathbf{S}$, proofs and provability of formulas are defined in the usual way, and by its logic, $\mathsf{S}$, we mean the set of provable formulas in it, i.e., all formulas $\phi$ such that $(\Rightarrow \phi)$ is provable in $\mathbf{S}$. The systems $\mathbf{FL_S}$ for $S \subseteq \{e, i, o, c\}$ are called basic substructural sequent calculi and their associated logics, basic substructural logics.

\begin{rem} \label{RemImplication}
Note that if $e \in S$, it is easy to show that in the system $\mathbf{FL_S}$, the formulas $\psi / \phi$ and $\phi \setminus \psi$ are provably equivalent. Hence, we can denote them by the common notation $\phi \to \psi$, using the usual connective $\to$. Moreover, it is also possible to axiomatize the system $\mathbf{FL_S}$ over the language $\mathcal{L}^*-\{/, \setminus\} \cup \{\to\}$, using all the rules in $\mathbf{FL_S}$, replacing the non-commutative implication rules with the commutative ones. Similarly, in the sequent calculus $\mathbf{FL_{ecw}}$, the formulas $\phi * \psi$ and $\phi \wedge \psi$ become equivalent and $0$ and $1$ also play the roles of $\bot$ and $\top$, respectively. Hence, it is possible to axiomatize $\mathbf{FL_{ecw}}$ over the language $\mathcal{L}= \{\wedge, \vee, \to, \top, \bot\}$, using all the initial sequents (except the ones for $0$ and $1$) and rules for the corresponding connectives. This is nothing but the usual sequent calculus $\mathbf{LJ}$, for the intuitionistic logic, $\mathsf{IPC}$. A similar type of argument also applies to $\mathbf{CFL_{ecw}}=\mathbf{LK}$, where $\mathbf{LK}$ is the sequent calculus for the classical logic, $\mathsf{CPC}$. Finally, it is worth mentioning that the logic $\mathsf{CFL_e}$ is essentially equivalent to the multiplicative additive linear logic, $\mathsf{MALL}$, introduced by Girard \cite{Girard} and the logic $\mathsf{FL_e}$ is known as its intuitionistic version, called $\mathsf{IMALL}$. $\mathsf{CFL_{ew}}$ is sometimes called the monoidal logic and $\mathsf{CFL_{ec}}$ is essentially equivalent to the relevant logic $\mathsf{R}$ without the distributive law. For more details, see \cite{Ono}.
\end{rem}

\begin{figure} \label{figure}
\centering
\begin{tikzpicture}
  [scale=.8,auto=right]
  \node (n1) at (11,5) {$\mathbf{FL}$};
  \node (n2) at (14,7)  {$\mathbf{FL_c}$};
  \node (n3) at (12,7)  {$\mathbf{FL_e}$};
  \node (n4) at (10,7)  {$\mathbf{FL_o}$};
  \node (n5) at (8,7)  {$\mathbf{FL_i}$};
  \node (n6) at (7,9)  {$\mathbf{FL_w}$};
  \node (n7) at (9,9)  {$\mathbf{FL_{ei}}$};
  \node (n8) at (11,9)  {$\mathbf{FL_{eo}}$};
  \node (n9) at (13,9)  {$\mathbf{FL_{co}}$};
  \node (n10) at (15,9)  {$\mathbf{FL_{ec}}$};
  \node (n11) at (8,11)  {$\mathbf{FL_{ew}}$};
  \node (n12) at (11,11)  {$\mathbf{FL_{ci}}= \mathbf{FL_{eci}}$};
  \node (n13) at (14,11)  {$\mathbf{FL_{eco}}$};
  \node (n14) at (11,13)  {$\mathbf{FL_{cw}}= \mathbf{FL_{ecw}}= \mathbf{LJ}$};
  \foreach \from/\to in {n1/n2,n1/n3,n1/n4,n1/n5,n4/n6,n5/n6,n5/n7,n3/n7,n3/n8,n4/n8,n4/n9,n2/n9,n3/n10,n2/n10,n6/n11,n7/n11,n8/n11,n7/n12,n10/n12,n8/n13,n9/n13,n10/n13,n11/n14,n12/n14,n13/n14}
    \draw (\from) -- (\to);

\end{tikzpicture}
\caption{Basic substructural calculi}
\end{figure}
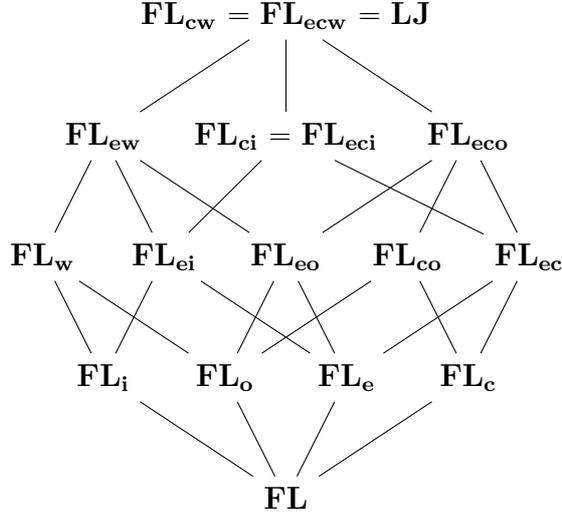

The fact that the sequent calculi $\mathbf{FL_S}$ and $\mathbf{CFL_{S'}}$ enjoy cut elimination for any $S, S' \subseteq \{e, i, o, c\}$ such that $e \in S'$,  has been shown independently by several authors. For instance, see \cite{Girard}, \cite{Lambeck}, and \cite{Ono90}.\\

So far, we have defined some basic substructural logics with their sequent calculi. Later in Definition \ref{DfnSubstructuralLogic}, we will introduce a substructural logic in its most general sense. Since the corresponding relation $\vdash_{\mathsf{L}}$, for a substructural logic $\mathsf{L}$ will be defined as an extension of $\vdash_{\mathsf{FL}}$, we will first define $\vdash_{\mathsf{FL}}$.

\begin{dfn} \label{DfnDeducibilityFL} 
We say a formula $\phi$ is \emph{provable} from a set of formulas $\Upsilon$ in the logic $\mathsf{FL}$ and we denote it by $\Upsilon \vdash_{\mathsf{FL}} \phi$, when the sequent $\Rightarrow \phi$ is provable in the sequent calculus $\mathbf{FL}$ by adding all $\Rightarrow \gamma$ as initial sequents for $\gamma \in \Upsilon$, i.e., $\{\Rightarrow \gamma\}_{\gamma \in \Upsilon} \vdash_{\mathbf{FL}} \Rightarrow \phi$. When $\Upsilon$ is the empty set we sometimes write $\mathsf{FL} \vdash \phi$ for $\vdash_{\mathsf{FL}} \phi$.
\end{dfn}

If the sequent $\phi_1, \ldots, \phi_n \Rightarrow \psi$ is provable in the sequent calculus $\mathbf{FL}$, then we have $\{\phi_1, \ldots, \phi_n \} \vdash_{\mathsf{FL}} \psi$. However, the converse, which is the deduction theorem, does not hold. In fact, unlike the classical and intuitionistic logics, most other substructural logics, including $\mathsf{FL}$, do not have a deduction theorem. We will see in Theorem \ref{ThmParametrizedDeduction} that only a restricted version of the deduction theorem (called the parametrized local deduction theorem) holds for $\vdash_{\mathsf{FL}}$. However, note that by definition, for a formula $\phi$ we have $\vdash_{\mathsf{FL}} \phi$ if and only if $\Rightarrow \phi$ is provable in the sequent calculus $\mathbf{FL}$. From now on, when no confusion occurs, we will write the fusion $\phi * \psi$ as $\phi \psi$.

\begin{dfn} \label{DfnSubstructuralLogic}
A set $\mathsf{L}$ of $\mathcal{L}^*$-formulas is called a \emph{substructural logic} (over $\mathsf{FL}$) if it is closed under substitution and satisfies the following conditions:
%Let $\mathsf{L}$ be a set of $\mathcal{L}^*$-formulas. $\mathsf{L}$ is a \emph{substructural logic} (over $\mathsf{FL}$) if it is closed under substitution, i.e., a homomorphism from the algebra of all formulas to itself, and satisfies the following conditions:
\begin{itemize}
\item[$(i)$]
$\mathsf{L}$ includes all formulas in $\mathsf{FL}$,
\item[$(ii)$]
if $\phi, \psi \in \mathsf{L}$, then $\phi \wedge \psi \in \mathsf{L}$,
\item[$(iii)$]
if $\phi, \phi \setminus \psi \in \mathsf{L}$, then $\psi \in \mathsf{L}$,
\item[$(iv)$]
if $\phi \in \mathsf{L}$ and $\psi$ is an arbitrary formula, then $\psi \setminus \phi \psi, \psi \phi / \psi \in \mathsf{L}$.
\end{itemize}
For a set of formulas $\Upsilon \cup \{\phi\}$, define $\Upsilon \vdash_{\mathsf{L}} \phi$ as $\Upsilon \cup \mathsf{L} \vdash_{\mathsf{FL}} \phi$. It is easy to see that $\vdash_{\mathsf{L}} \phi$ is equivalent to $\phi \in \mathsf{L}$, for any substructural logic $\mathsf{L}$.
\end{dfn}

When $\mathsf{L}$ is the logic $\mathsf{FL}$, then $\vdash_{\mathsf{FL}}$ defined in Definition \ref{DfnSubstructuralLogic} will be the same as the one defined in Definition \ref{DfnDeducibilityFL}. Therefore, there will be no ambiguity. As a corollary of Theorem \cite[2.16]{Ono}, it is shown that Definition \ref{DfnSubstructuralLogic} can be replaced by the following: a substructural logic over $\mathsf{FL}$ is a set of formulas closed under both substitution and $\vdash_{\mathsf{FL}}$.\\
It is easy to see that for any subset $S$ of $\{e, i, o, c\}$, the logic $\mathsf{FL_S}$ is a substructural logic  as in Definition \ref{DfnSubstructuralLogic}. Moreover, since for all the sequent calculi in Table \ref{table:pekh}, the sequent calculus $\mathbf{FL}$ and specially the cut rule is present, all their corresponding logics are closed under the conditions in Definition \ref{DfnSubstructuralLogic}. Therefore, they are all substructural logics.\\
Finally, we can see that for (possibly empty) sequences of formulas $\Sigma= \sigma_1, \ldots, \sigma_m$ and $\Gamma$ 
\[
\Rightarrow \sigma_1, \ldots, \Rightarrow \sigma_m \vdash_{\mathbf{FL_S}} \Gamma \Rightarrow \phi \;\; \textit{implies} \;\; \Sigma, \Gamma \vdash_{\mathsf{FL_S}} \phi.
\]
This can be easily shown since we can simulate each rule in $\{e, i, o, c\}$, using its corresponding axiom (defined below) and the cut rule:
\[
(e): (\phi  \psi) \setminus (\psi  \phi) \;\;, \;\; (c): \phi \setminus (\phi  \phi) \;\;, \;\; (i): \phi \setminus 1 \;\;, \;\; (o): 0 \setminus \phi
\]

\begin{dfn} \label{DfnConjugate}
Let $\phi$ and $\alpha$ be formulas. Define
\[
\lambda_{\alpha}(\phi)= (\alpha \setminus (\phi \alpha)) \wedge 1 \;\;\;\; \textit{and} \;\;\;\;
\rho_{\alpha}(\phi) = ((\alpha \phi) / \alpha) \wedge 1.
\]
We call $\lambda_{\alpha}(\phi)$ and $\rho_{\alpha}(\phi)$ the \emph{left and right conjugate of $\phi$ with respect to $\alpha$}, respectively. An \emph{iterated conjugate} of $\phi$ is a composition $\gamma_{\alpha_1} (\gamma_{\alpha_2}(\ldots \gamma_{\alpha_n}(\phi)))$, for formulas $\alpha_1, \ldots, \alpha_n$ where $n \geq 0$ and $\gamma_{\alpha_i} \in \{\lambda_{\alpha_i}, \rho_{\alpha_i} \}$. 
\end{dfn}

It is shown in \cite[Lemma 2.13.]{Ono} that if a sequent $\Gamma, \alpha, \beta, \Sigma \Rightarrow \phi$ is provable in $\mathbf{FL}$, then the following sequents are also provable in $\mathbf{FL}$:
\[
\Gamma , \beta , \lambda_{\beta}(\alpha), \Sigma \Rightarrow \phi \;\;\;\; \textit{and} \;\;\;\; \Gamma , \rho_{\alpha}(\beta), \alpha, \Sigma \Rightarrow \phi.
\]
The following theorem states the parametrized local deduction theorem for $\mathsf{FL}$.

\begin{thm} \cite[Theorem 2.14.]{Ono} \label{ThmParametrizedDeduction}
Let $\mathsf{L}$ be a substructural logic and $\Upsilon \cup \Xi \cup \{\phi\}$ be a set of formulas. Then,
\[
\Upsilon, \Xi \vdash_{\mathsf{L}} \phi \;\;\;\; \textit{iff} \;\;\;\; \Upsilon \vdash_{\mathsf{L}} (\bigast_{i=1}^{n} \gamma_{i} (\psi_i)) \setminus \phi
\]
for some n, where each $\gamma_{i} (\psi_i)$ is an iterated conjugate of a formula $\psi_i \in \Xi$.
\end{thm}

\begin{rem}
Note that the definition of $\vdash_{\mathsf{L}}$ in Definition \ref{DfnSubstructuralLogic} depends on the sequent calculus $\mathbf{FL}$ and not the mere logic $\mathsf{FL}$. Because it uses $\vdash_{\mathsf{FL}}$, defined in Definition \ref{DfnDeducibilityFL}, which itself uses the sequent calculus $\mathbf{FL}$. It is possible to use Theorem \ref{ThmParametrizedDeduction} to provide the following proof system-independent definition of $\vdash_{\mathsf{L}}$:
\[
\Upsilon \vdash_{\mathsf{L}} \phi \;\;\;\;\; \textit{iff} \;\;\;\;\; (\bigast_{i=1}^{n} \gamma_{i} (A_i)) \setminus \phi \in \mathsf{L} \;\;\;\;\; \textit{iff} \;\;\;\;\; (\bigast_{i=1}^{m} \gamma_{i} (B_i)) \setminus \phi \in \mathsf{FL}
\]
for some $n$ and $m$  where each $\gamma_{i} (A_i)$ is an iterated conjugate of some $A_i \in \Upsilon$ and $\gamma_{i} (B_i)$ an iterated conjugate of some $B_i \in \Upsilon \cup \{\mathsf{L}\}$.
\end{rem}

\subsection{Super-basic logics}
In \cite{Visser}, Visser introduced basic propositional logic, $\mathsf{BPC}$, and formal propositional logic, $\mathsf{FPL}$, to interpret implication as formal provability. In \cite{Wim}, Ruitenberg reintroduced $\mathsf{BPC}$ from a more philosophical standpoint and extended the system to its predicate version, $\mathsf{BQC}$. In the following, we present the logic $\mathsf{BPC}$ using its sequent calculus, $\mathbf{BPC}$, introduced in \cite{Basic}. The main difference between this logic and the intuitionistic logic is that the modus ponens rule is weakened here and hence $\mathsf{BPC}$ is weaker than the intuitionistic logic, $\mathsf{IPC}$.\\

The language of $\mathbf{BPC}$ is $\mathcal{L}=\{\wedge, \vee, \top, \bot, \to\}$ and negation is defined as the abbreviation for $\neg \phi= \phi \to \bot$. Recall that $\Upsilon$ and $\Xi$ are the reserved notations for multisets of formulas. By $\Upsilon, \phi$ or $\phi, \Upsilon$, we mean the multiset $\Upsilon \cup \{\phi\}$. Sequents of $\mathbf{BPC}$ are of the same form as the sequents of $\mathbf{LK}$, except that they are written for multisets of formulas, rather than sequences of formulas, and they are interpreted in the same way, i.e., $I(\Upsilon \Rightarrow \Xi)=\bigwedge \Upsilon \to \bigvee \Xi$ where $\bigwedge \emptyset = \top$ and $\bigvee \emptyset = \bot$. The initial sequents and rules of $\mathbf{BPC}$ are as follows:
\begin{center}
$\Upsilon, \phi\Rightarrow \phi, \Xi  $ \;\;\;\;\;  $\Upsilon \Rightarrow \top, \Xi $ \;\;\;\;\; $\Upsilon, \bot \Rightarrow \Xi$
\end{center}

\begin{center}
 \begin{tabular}{c c} 
 \AxiomC{$\phi, \psi, \Upsilon \Rightarrow \Xi$}
 \RightLabel{$(L \wedge)$} 
 \UnaryInfC{$\phi \wedge \psi, \Upsilon \Rightarrow \Xi$}
 \DisplayProof
 &
  \AxiomC{$\Upsilon \Rightarrow \Xi, \phi$}
 \AxiomC{$\Upsilon \Rightarrow \Xi, \psi$}
 \RightLabel{$(R \wedge)$} 
 \BinaryInfC{$\Upsilon \Rightarrow \Xi, \phi \wedge \psi$}
 \DisplayProof
\end{tabular}
\end{center}

\begin{center}
 \begin{tabular}{c c}
 \AxiomC{$\phi, \Upsilon \Rightarrow \Xi$}
 \AxiomC{$\psi , \Upsilon \Rightarrow \Xi$}
 \RightLabel{$(L \vee)$} 
 \BinaryInfC{$ \phi \vee \psi, \Upsilon \Rightarrow \Xi$}
 \DisplayProof
 &
 \AxiomC{$\Upsilon \Rightarrow \Xi, \phi, \psi$}
 \RightLabel{$(R \vee)$} 
 \UnaryInfC{$\Upsilon \Rightarrow \Xi, \phi \vee \psi$}
 \DisplayProof
\end{tabular}
\end{center}

\begin{center}
 \begin{tabular}{c}
 \AxiomC{$\phi, \Upsilon \Rightarrow \psi$}
 \RightLabel{$(R \to)$} 
 \UnaryInfC{$\Upsilon \Rightarrow \Xi, \phi \to \psi$}
 \DisplayProof
\end{tabular}
\end{center}

\begin{center}
 \begin{tabular}{c c}
  \AxiomC{$\phi \wedge \psi , \Upsilon \Rightarrow \Xi$}
   \AxiomC{$\phi \wedge \theta , \Upsilon \Rightarrow \Xi$}
 \RightLabel{$(D)$} 
 \BinaryInfC{$\phi \wedge (\psi \vee \theta), \Upsilon \Rightarrow \Xi$}
 \DisplayProof
 &
  \AxiomC{$\Upsilon \Rightarrow \phi \to \psi$}
   \AxiomC{$\Upsilon \Rightarrow \psi \to \theta$}
 \RightLabel{$(Tr)$} 
 \BinaryInfC{$\Upsilon \Rightarrow \Xi, \phi \to \theta$}
 \DisplayProof
\end{tabular}
\end{center}

\begin{center}
 \begin{tabular}{c c}
  \AxiomC{$\Upsilon \Rightarrow \phi \to \psi$}
   \AxiomC{$\Upsilon \Rightarrow \phi \to \theta$}
 \RightLabel{$(F \wedge)$} 
 \BinaryInfC{$\Upsilon \Rightarrow \Xi, \phi \to (\psi \wedge \theta)$}
 \DisplayProof
 &
  \AxiomC{$\Upsilon \Rightarrow \phi \to \theta$}
   \AxiomC{$\Upsilon \Rightarrow \psi \to \theta$}
 \RightLabel{$(F \vee)$} 
 \BinaryInfC{$\Upsilon \Rightarrow \Xi, (\phi \vee \psi) \to \theta$}
 \DisplayProof
\end{tabular}
\end{center}

\begin{center}
 \begin{tabular}{c}
  \AxiomC{$\Upsilon_1 \Rightarrow \phi , \Xi_1$}
   \AxiomC{$\Upsilon_2 , \phi \Rightarrow \Xi_2$}
 \RightLabel{$(cut)$} 
 \BinaryInfC{$\Upsilon_1, \Upsilon_2 \Rightarrow \Xi_1, \Xi_2$}
 \DisplayProof
\end{tabular}
\end{center}

Note that in this proof system we are using multisets of formulas, and hence the exchange rules are built in. Moreover, the left and right weakening and contraction rules are also admissible and the system enjoys the cut elimination (see \cite{Basic} Lemma 2.2, Lemma 2.12, Lemma 2.14, and Theorem 2.17, respectively).\\
An extension of $\mathbf{BPC}$ augmented by the initial sequent $\top \to \bot \Rightarrow \bot$ is introduced in \cite{EBPC} and is denoted by $\mathbf{EBPC}$. We denote the corresponding logics of $\mathbf{BPC}$ and $\mathbf{EBPC}$ by $\mathsf{BPC}$ and $\mathsf{EBPC}$, respectively. It is shown that $\mathsf{BPC} \subsetneqq \mathsf{EBPC} \subsetneqq \mathsf{IPC}$ \cite{EBPC}. We will not use the semantical characterizations of the logics $\mathsf{BPC}$ or $\mathsf{EBPC}$. However, for the curious reader, it is worth mentioning that
$\mathsf{BPC}$ is sound and complete with respect to transitive persistent Kripke models, while the logic $\mathsf{EBPC}$ is sound and complete with respect to rooted finite transitive persistent Kripke models with reflexive terminal nodes \cite{EBPC}. Moreover, $\mathsf{BPC}$ is connected with the modal logic $\mathbf{K4}$ via G\"{o}del's translation $T$, as shown in \cite{Visser}. %Since formulas $(A \to (A \to B)) \to (A \to B)$ and $(A \to (B \to C)) \to (B \to (A \to C))$ are not always true in transitive models (the former formula corresponds to the contraction rule and the latter to the exchange rule), one may view $\mathsf{BPC}$ and $\mathsf{EBPC}$ as substructural logics.

\begin{dfn}
We say a formula $\phi$ is \emph{provable} from a set of formulas $\Upsilon$ in the logic $\mathsf{BPC}$ and we denote it by $\Upsilon \vdash_{\mathsf{BPC}} \phi$, when the sequent $\Rightarrow \phi$ is provable in the sequent calculus $\mathbf{BPC}$ augmented by $\Rightarrow \gamma$ for all $\gamma \in \Upsilon$ as initial sequents. 
\end{dfn}

The following remark is an important property of the logic $\mathsf{BPC}$ concerning the modus ponens rule.

\begin{rem} \label{RemarkBPC}
Note that although the modus ponens rule in the form
\begin{center}
 \begin{tabular}{c}
  \AxiomC{$\Upsilon \Rightarrow \phi$}
   \AxiomC{$\Upsilon \Rightarrow \phi \to \psi$} 
 \BinaryInfC{$\Upsilon \Rightarrow \psi$}
 \DisplayProof
\end{tabular}
\end{center}
is neither present nor admissible in the sequent calculus $\mathbf{BPC}$, a simplified version of it, where $\Upsilon$ is the empty set, namely
\begin{center}
 \begin{tabular}{c}
  \AxiomC{$ \Rightarrow \phi$}
   \AxiomC{$ \Rightarrow \phi \to \psi$} 
 \BinaryInfC{$ \Rightarrow \psi$}
 \DisplayProof
\end{tabular}
\end{center}
is admissible (but not provable) in $\mathbf{BPC}$. Therefore, the logic $\mathsf{BPC}$ admits the modus ponens rule, i.e., if $\phi \in \mathsf{BPC}$ and $\phi \to \psi \in \mathsf{BPC}$, then $\psi \in \mathsf{BPC}$. The reason is that if $\phi \to \psi \in \mathsf{BPC}$ then $\mathbf{BPC} \vdash (\Rightarrow \phi \to \psi)$. By cut elimination, there exists a cut-free proof of $(\Rightarrow \phi \to \psi)$ in  $\mathbf{BPC}$.
Then by induction on the structure of this cut-free proof we can show that $\mathbf{BPC} \vdash \phi \Rightarrow \psi$. Finally, since $\phi \in \mathsf{BPC}$, we have $\mathbf{BPC} \vdash (\Rightarrow \phi)$, and then using the cut rule we get $\mathbf{BPC} \vdash (\Rightarrow \psi)$, which means $\psi \in \mathsf{BPC}$. However, we have $\Rightarrow \phi, \Rightarrow \phi \to\psi \nvdash_{\mathbf{BPC}} \Rightarrow \psi$, which means that modus ponens is not provable in $\mathbf{BPC}$. The same property also holds for the logic $\mathsf{EBPC}$ \cite[Proposition 3.11]{MWim}.
\end{rem}

In the following, we will define a super-basic logic using its consequence relation. First, let us recall the definition of a consequence relation.

\begin{dfn} \label{DfnConsequenceRelation}
A \emph{consequence relation} $\vdash_{\mathsf{L}}$ is a relation between sets of formulas and formulas in the language of the logic $\mathsf{L}$, such that the following conditions hold:
\begin{itemize}
\item
it is closed under substitution;
\item 
if $\phi \in \Upsilon$ then $\Upsilon \vdash_{\mathsf{L}} \phi$;
\item
if $\Upsilon \vdash_{\mathsf{L}} \phi$ and for every $\gamma \in \Upsilon$ we have $\Xi \vdash_{\mathsf{L}} \gamma$, then $\Xi \vdash_{\mathsf{L}} \phi$;
\item
if $\Upsilon \vdash_{\mathsf{L}} \phi$ then there exists a finite set $\Upsilon' \subseteq \Upsilon$ such that $\Upsilon' \vdash_{\mathsf{L}} \phi$.
\end{itemize}
Logic of a consequence relation $\vdash_{\mathsf{L}}$, denoted by $\mathsf{L}$, is defined as the set of formulas $\phi$ such that $\vdash_{\mathsf{L}} \phi$.
\end{dfn}

The relation $\vdash_{\mathsf{BPC}}$ defined before is an example of a consequence relation.

\begin{dfn} \label{DfnSuperBasicConsequenceRelation}
By a \emph{super-basic consequence relation}, $\vdash_{\mathsf{L}}$, we mean a consequence relation for formulas in the language $\mathcal{L}$, which is closed under the rules and initial sequents of the sequent calculus $\mathbf{BPC}$, i.e., 
\begin{itemize}
\item
for any initial sequent of $\mathbf{BPC}$ of the form $\Upsilon \Rightarrow \Xi$, we have $\Upsilon \vdash_{\mathsf{L}} \bigvee \Xi$,
\item
for any rule of $\mathbf{BPC}$ with premises $\Upsilon_i \Rightarrow \Xi_i$ for $i \in \{1,2\}$ and the conclusion $\Upsilon \Rightarrow \Xi$, if $\Upsilon_i \vdash_{\mathsf{L}} \bigvee \Xi_i$ for each premise, then $\Upsilon \vdash_{\mathsf{L}} \bigvee \Xi$.
\end{itemize}
A consequence relation is called \emph{super-intuitionistic} when it is closed under the rules and initial sequents of the sequent calculus $\mathbf{LJ}$. The logic of a super-basic (super-intuitionistic) consequence relation is called a super-basic (super-intuitionistic) logic.
\end{dfn}

Clearly, $\vdash_{\mathsf{BPC}}$ and $\vdash_{\mathsf{EBPC}}$ are super-basic consequence relations, and $\mathsf{BPC}$ and $\mathsf{EBPC}$ are super-basic logics. It is easy to see that any super-intuitionistic consequence relation (logic) is also a super-basic one. Moreover, if $\Upsilon \vdash_{\mathsf{BPC}} \phi$ with the proof $\pi$, then for any super-basic logic $\mathsf{L}$ we have $\Upsilon \vdash_{\mathsf{L}} \phi$ with the same proof $\pi$. The same property holds for $\mathsf{IPC}$ and any super-intuitionistic logic.

\begin{rem} 
We have defined a substructural logic and a super-basic logic in different ways in Definitions \ref{DfnSubstructuralLogic} and \ref{DfnSuperBasicConsequenceRelation}, respectively. In the former, a substructural logic is defined as a set of formulas satisfying some conditions and in the latter, a super-basic logic is defined by its consequence relation. The reason is that in the absence of the modus ponens rule, defining a logic as a set of formulas or using its consequence relation to define it, may result in different sets. The latter is a primitive notion for logics lacking the modus ponens rule, because these logics are usually defined by their sequent calculi. Therefore, we chose to define a super-basic logic by its cosequence relation.
%A reasonable way to define a super-basic logic $\mathsf{L}$, similar to Definition \ref{DfnSubstructuralLogic}, is taking a set of formulas closed under substitution and containing all the formulas in the logic $\mathsf{BPC}$. We would like $\mathsf{IPC}$ to be a super-basic logic. However, the observation that $\phi, \phi \to \psi \vdash_{\mathbf{LJ}} \psi$ but $\phi, \phi \to \psi, \mathsf{IPC} \nvdash_{\mathbf{BPC}} \psi$ \textcolor{red}{chejoori hamchi edeaE ro mishe sabet kard?} shows that defining a super-basic logic in this way makes some problems. Even adding the rule $mp$ to the definition of a super-basic logic $\mathsf{L}$, that is if $\phi, \phi \to \psi \in \mathsf{L}$ then $\psi \in \mathsf{L}$, does not help to overcome the obstacle illuminated in the observation. 
\end{rem}

For a set of formulas $\Upsilon$ consider the sequent calculus derived by adding $\Rightarrow \gamma$ for all $\gamma \in \Upsilon$ to $\mathbf{LJ}$ as initial sequents and denote it by $\mathbf{LJ}_{\Upsilon}$. For a super-intuitionistic logic $\mathsf{L}$ by $\mathsf{L} + \Upsilon$, we mean the logic of the least consequence relation closed under the rules and initial sequents of $\mathbf{LJ}_{\Upsilon}$. We can define Jankov's logic, $\mathsf{KC}$ by adding the weak excluded middle formula to the intuitionistic logic, i.e., $\mathsf{KC}=\mathsf{IPC} + \neg p \vee \neg \neg p$. Although we do not use Kripke models in this paper, for the interested reader we mention that the condition on the Kripke models for this logic is being directed.
The axioms $BD_n$ are defined in the following way:
\[
BD_0 := \bot \;\;\; , \;\;\; BD_{n+1} := p_n \vee (p_n \to BD_n). 
\]
The logic of bounded depth $\mathsf{BD_n}$ is then defined as $\mathsf{IPC + } BD_n$. 
Define the logic $\mathsf{T_k}$ as 
\[
\mathsf{IPC} + \bigwedge_{i=0}^{k}((p_i \to \bigvee_{j \neq i} p_j) \to \bigvee_{j}p_j) \to \bigvee_i p_i.
\]
A super-intuitionistic logic $\mathsf{L}$ has branching $k$ if $\mathsf{T_k} \subseteq \mathsf{L}$. We say a super-intuitionistic logic $\mathsf{L}$ has finite branching if there exists a number $k$ such that $\mathsf{L}$ has branching less than or equal to $k$, otherwise we call it infinite branching. The following theorem by Je\v{r}\'{a}bek obtains a nice characterization of super-intuitionistic logics of infinite branching. Since the results in this paper concern these logics, we will present the theorem, although we will not use it in any of our future discussions. 

\begin{thm} \cite[Theorem 6.9]{Jerabek} \label{ThmInfiniteBranching} 
Let $\mathsf{L}$ be a super-intuitionistic logic. Then, $\mathsf{L}$ has infinite branching if and only if $\mathsf{L} \subseteq \mathsf{BD_2}$ or $\mathsf{L} \subseteq \mathsf{KC+BD_3}$.
\end{thm}

\section{Frege and extended Frege systems} \label{SectionFregeExtendedFrege}
The purpose of this section is to introduce Frege and extended Frege systems for substructural and super-basic logics. For this matter, we will recall and generalize some basic concepts in proof complexity. For more background the reader may consult \cite{KrajicekProof}.

\begin{dfn} \label{dfn1}
Let $\mathsf{L}$ be a set of finite strings over a finite alphabet. A \emph{(propositional) proof system for $\mathsf{L}$} is a polynomial-time function $\mathbf{P}$ with the range $\mathsf{L}$. Any string $\pi$ such that $\mathbf{P}(\pi)= \phi$ is a \emph{$\mathbf{P}$-proof} of the string $\phi$, also denoted by $\mathbf{P} \vdash^{\pi} \phi$.
We denote proof systems by bold-face capital Roman letters $\mathbf{P}, \mathbf{Q}$.
\end{dfn}
We often consider proof systems for a logic $\mathsf{L}$. %with the language $\mathcal{L}_{\mathsf{L}}$.  
The usual Hilbert-style systems with finitely many axiom schemes and Gentzen's sequent calculi are instances of propositional proof systems for their corresponding logics, because in polynomial time one can simply decide whether a finite string is a proof in the system or not.\\
By the length of a formula $\phi$, or a proof $\pi$, we mean the number of symbols it contains and we denote them by $|\phi|$ and $|\pi|$, respectively. 

\begin{dfn} \label{DfnProofSystemSimulate}
Let $\mathbf{P}$ and $\mathbf{Q}$ be two proof systems for logics $\mathsf{L}_{\mathbf{P}}$ and $\mathsf{L}_{\mathbf{Q}}$ with the languages $\mathcal{L}_{\mathbf{P}}$ and $\mathcal{L}_{\mathbf{Q}}$, respectively. Let $tr$ be a polynomial-time translation function from the strings in the language $\mathcal{L}_{\mathbf{P}}$ to the strings in the language $\mathcal{L}_{\mathbf{Q}}$.
\begin{itemize}
\item[$(i)$]
We say \emph{$\mathbf{Q}$ is at least as strong as $\mathbf{P}$ with respect to $tr$}, iff for any $\pi$ there exists $\pi'$ such that $tr(\mathbf{P}(\pi))= \mathbf{Q}(\pi')$. 
\item[$(ii)$]
We say that \emph{$\mathbf{Q}$ simulates $\mathbf{P}$} (or \emph{$\mathbf{P}$ is simulated by $\mathbf{Q}$}) \emph{with respect to $tr$}, denoted by $\mathbf{P} \leq^{tr} \mathbf{Q}$, iff there exists a polynomial function $g : N \to N$  such that for any $\pi$ %and $\phi$ such that $\mathbf{Q}(\pi)= \phi$ 
there exists $\pi'$ for which we have $tr(\mathbf{P}(\pi))= \mathbf{Q}(\pi')$ and $|\pi'| \leq g(|\pi|)$.
\item[$(iii)$]
By $\mathsf{L}_{\mathbf{P}} \subseteq^{tr} \mathsf{L}_{\mathbf{Q}}$, we mean that for any $\phi \in \mathsf{L}_{\mathbf{P}}$ we have $tr(\phi) \in \mathsf{L}_{\mathbf{Q}}$.
\end{itemize}
\end{dfn}
In the parts $(i)$ and $(ii)$ of the definition, for the simpler case when $\mathsf{L}_{\mathbf{P}} \subseteq \mathsf{L}_{\mathbf{Q}}$ and the translation function is the inclusion function, we drop the phrase ``with respect to $tr$", and we simply write $\mathbf{P} \leq \mathbf{Q}$ in $(ii)$.
In the case that $\mathsf{L}_{\mathbf{P}} = \mathsf{L}_{\mathbf{Q}}$ and the translation function is the identity function, we say that the proof systems $\mathbf{P}$ and $\mathbf{Q}$ are equivalent when they simulate each other. \\

In the following, we present a translation function $t$ that enables us to translate $\mathcal{L^*}$ into $\mathcal{L}$ and hence to carry out results in systems with the language $\mathcal{L}$ to systems with the language $\mathcal{L}^*$. This translation function has a similar effect as bringing back the structural rules to the systems:

\begin{dfn} \label{DfnTranslation}
Define the function $t: \mathcal{L}^* \to \mathcal{L}$ as follows:
\begin{itemize}
\item[$\bullet$]
$p^t=p$, where $p$ is a propositional variable;
\item[$\bullet$]
$0^t= \bot$, $1^t= \top$;
\item[$\bullet$]
$(\phi \circ \psi)^t = \phi^t \circ \psi^t$, where $\circ \in \{\wedge, \vee\}$;
\item[$\bullet$]
$(\phi * \psi)^t= \phi^t \wedge \psi^t$;
\item[$\bullet$]
$(\psi / \phi)^t = (\phi \setminus \psi)^t = \phi^t \to \psi^t$.
\end{itemize}
For $\Gamma$, a finite sequence of formulas $\gamma_1, \gamma_2, \ldots, \gamma_n$, by $\Gamma^t$ we mean the sequence of formulas $\gamma_1^t, \gamma_2^t \ldots, \gamma_n^t$. It is easy to see that $|\phi^t|= |\phi|$.
\end{dfn}

The following lemma, which will be used in the future sections, is an example of how the translation $t$ works. It expresses the relation between sequents provable in the sequent calculus $\mathbf{WL}$ and the translated version of the sequents in the system $\mathbf{BPC}$.
\begin{lem} \label{LemTranslationWLBPC}
Let $\Gamma$ be a sequence of formulas and $A$ be a formula in the language of $\mathbf{WL}$, i.e., $\{1, \bot, \wedge, \vee, *, \setminus\}$. Then
\[
\mathbf{WL} \vdash \Gamma \Rightarrow A \;\; \textit{implies} \;\; \mathbf{BPC} \vdash \Gamma^t \Rightarrow A^t.
\]
\end{lem}

\begin{proof}
It follows by a straightforward induction on the structure of the proof of $\Gamma \Rightarrow A$ in $\mathbf{WL}$. Note that as mentioned earlier, the left contraction rule and both right and left weakening rules are derivable in $\mathbf{BPC}$ and the exchange rules are built in. The cases for the other rules are easy. As an example, suppose that the last rule in the proof of $\Gamma \Rightarrow A$ is $(R*)$:
\begin{center}
 \begin{tabular}{c}
  \AxiomC{$\Sigma \Rightarrow \phi$}
   \AxiomC{$\Pi \Rightarrow \psi$}
 \BinaryInfC{$\Sigma , \Pi \Rightarrow \phi * \psi$}
 \DisplayProof
\end{tabular}
\end{center}
where $\Gamma= \Sigma , \Pi$ and $A = \phi * \psi$. Then, by the induction hypothesis, we have $\mathbf{BPC} \vdash \Sigma^t \Rightarrow \phi^t$ and $\mathbf{BPC} \vdash \Pi^t \Rightarrow \psi^t$. Since the left weakening rule is admissible in $\mathbf{BPC}$, we have both $\mathbf{BPC} \vdash \Sigma^t, \Pi^t \Rightarrow \phi^t$ and $\mathbf{BPC} \vdash \Sigma^t, \Pi^t \Rightarrow \psi^t$. Using the rule $(R \wedge)$ we obtain $\mathbf{BPC} \vdash \Sigma^t, \Pi^t \Rightarrow \phi^t \wedge \psi^t$, which is what we wanted.
\end{proof}

\begin{rem} \label{RemLM}
For any substructural logic $\mathsf{L}$ and any super-intuitionistic logic $\mathsf{M}$, it is easy to see that $\mathsf{L} \subseteq^t \mathsf{M}$ implies the stronger form:
\[
\phi_1, \ldots, \phi_n \vdash_{\mathsf{L}} \phi \;\;\;  \textit{implies} \;\;\; \phi_1^t, \ldots, \phi_n^t \vdash_{\mathsf{M}} \phi^t.
\]
Here is the sketch of the proof. If there exists $i$ such that $\phi=\phi_i$, the proof is obvious. If $\phi \in \mathsf{L}$, the claim follows from $\mathsf{L} \subseteq^t \mathsf{M}$. If $\phi$ is derived by an $\mathbf{FL}$-rule \begin{tabular}{c}
  \AxiomC{$\psi_1$}
   \AxiomC{$\ldots$}
   \AxiomC{$\psi_m$}
 \TrinaryInfC{$\phi$}
 \DisplayProof
\end{tabular},
then it is easy to see that 
\begin{tabular}{c}
  \AxiomC{$\psi_1^t$}
   \AxiomC{$\ldots$}
   \AxiomC{$\psi_m^t$}
 \TrinaryInfC{$\phi^t$}
 \DisplayProof
\end{tabular}
is derivable in $\mathbf{LJ}$. And, by Definition \ref{DfnSuperBasicConsequenceRelation}, the consequence relation $\vdash_{\mathsf{M}}$ is closed under the rules and initial sequents of $\mathbf{LJ}$. 
\end{rem}

In the following, we will define Frege and extended Frege systems for substructural and super-basic logics.

\begin{dfn} \label{dfnInferenceSystem} 
An \emph{inference system} $\mathbf{P}$ is defined by a set of rules of the form 
\begin{center}
 \begin{tabular}{c}
  \AxiomC{$\phi_1$}
   \AxiomC{$\ldots$}
   \AxiomC{$\phi_m$}
 \TrinaryInfC{$\phi$}
 \DisplayProof
\end{tabular}
\end{center}
where for any $1 \leq i \leq m$, the formulas $\phi_i$ are called the \emph{premises} of the rule, and the formula $\phi$ is called its \emph{conclusion}. A rule with no premise is called an \emph{axiom}. A \emph{$\mathbf{P}$-proof}, $\pi$, of a formula $\phi$ from a set of formulas $X$ is defined as a sequence of formulas $\phi_1, \ldots , \phi_n=\phi$, where for $1 \leq i \leq n$, either $\phi_i \in X$ or $\phi_i$ is derived from some $\phi_j$'s, $j < i$, by a substitution instance of a rule of the system $\mathbf{P}$.  If the set $X$ is empty, then we say that the formula $\phi$ is provable in $\mathbf{P}$. Each $\phi_i$ is called a \emph{step} or a \emph{line} in the proof $\pi$. The number of lines of a proof $\pi$ is denoted by $\lambda(\pi)$ and it is clear that it is less than or equal to the length of the proof. The set of all provable formulas in $\mathbf{P}$ is called its logic.
If there is a $\mathbf{P}$-proof for $\phi$ from assumptions $\phi_1, \ldots, \phi_n$, we write $\phi_1, \ldots , \phi_n \vdash_{\mathbf{P}} \phi$. Specially, for every rule of the above form we have $\phi_1, \ldots, \phi_m \vdash_{\mathbf{P}} \phi$. 
\end{dfn}

There are two measures for the complexity of proofs in proof systems. The first one is the length of the proof and the other is the number of proof steps (also called proof-lines). This only makes sense for proof systems in which the proofs consist of lines containing formulas or sequents. Hilbert-style proof systems, Gentzen's sequent calculi, and Frege systems are examples of such proof systems.

\begin{dfn} \label{DfnFrege}
Let $\mathsf{L}$ and $\mathsf{M}$ be two substructural or two super-basic logics such that $\mathsf{L} \subseteq \mathsf{M}$. The inference system $\mathbf{P}$ is called a \emph{Frege system for $\mathsf{L}$ with respect to $\mathsf{M}$}, for short an $\mathsf{L}-\mathbf{F}$ system wrt $\mathsf{M}$, if it satisfies the following conditions:
\begin{itemize}
\item[$(1)$]
$\mathbf{P}$ has finitely many rules,
\item[$(2)$]
$\mathbf{P}$ is sound: if $\vdash_{\mathbf{P}} \phi$, then $\phi \in \mathsf{L}$,
\item[$(3)$]
$\mathbf{P}$ is strongly complete: if $\phi_1, \ldots, \phi_n \vdash_{\mathsf{L}} \phi$, then $\phi_1, \ldots, \phi_n \vdash_{\mathbf{P}} \phi$.
\item[$(4)$]
every rule in $\mathbf{P}$ is $\mathsf{M}$-standard: if 
 \begin{tabular}{c}
  \AxiomC{$\phi_1$}
   \AxiomC{$\ldots$}
   \AxiomC{$\phi_m$}
 \TrinaryInfC{$\phi$}
 \DisplayProof
\end{tabular} 
is a rule in $\mathbf{P}$, then $\phi_1, \ldots , \phi_m \vdash_{\mathsf{M}} \phi$.
\end{itemize}
In the case that $\mathsf{L} = \mathsf{M}$, we simply call this system a \emph{Frege system for $\mathsf{L}$}.
\end{dfn}
Here are some remarks.
Using the condition $(4)$, it is easy to see that any Frege system $\mathbf{P}$ for $\mathsf{L}$ wrt $\mathsf{M}$ has the property that if $\phi_1, \ldots, \phi_n \vdash_{\mathbf{P}} \phi$, then $\phi_1, \ldots, \phi_n \vdash_{\mathsf{M}} \phi$.
For a substructural logic $\mathsf{L}$, we will only consider Frege systems for $\mathsf{L}$ wrt $\mathsf{L}$, i.e., $\mathsf{M}=\mathsf{L}$. For $S$ a subset of $\{e, i, o, c\}$, the Hilbert-style proof system $\mathbf{HFL_S}$ is an example of a Frege system for the basic substructural logic $\mathsf{FL_S}$ (see \cite[Section 2.5]{Ono}). The usual Hilbert-style systems for classical and intuitionistic logics, $\mathbf{HK}$ and $\mathbf{HJ}$, are also examples of Frege systems for $\mathsf{CPC}$ and $\mathsf{IPC}$, respectively; see \cite[Sections 1.3.1 and 1.3.3]{Ono}. \\
Usually, a Frege system for a logic $\mathsf{L}$ is defined by some $\mathsf{L}$-standard rules in the sense of the condition $(4)$ in Definition \ref{DfnFrege}. This condition is useful to establish the uniqueness of these systems up to equivalence (see Lemma \ref{LemEquivalenceOfEveryFrege}). However, in this paper we generalize the usual definition to add another and possibly stronger logic $\mathsf{M}$ as a parameter to control the derivability of the rules. The logic $\mathsf{M}$ is not necessarily equal to $\mathsf{L}$. The reason for this choice is the somehow strange behaviour of some Hilbert-style proof systems for some super-basic logics. For instance, any natural Hilbert-style system for $\mathsf{BPC}$ includes the modus pones rule (see for instance Theorem \ref{ThmFregeForBPC}). While this rule is admissible in $\mathsf{BPC}$ and hence harmless to the soundness of the system, it can not be derivable inside the logic $\mathsf{BPC}$ itself, i.e, $\phi , \phi \to \psi \nvdash_{\mathsf{BPC}} \psi$. Therefore, the modus ponens rule violates the $\mathsf{BPC}$-standradness condition. To address such systems, it may be reasonable to relax the $\mathsf{BPC}$-standardness condition a bit to also include the modus ponense rule. The smallest logic containing $\mathsf{BPC}$ and modus ponens is $\mathsf{IPC}$ and hence we have to pick $\mathsf{M}=\mathsf{IPC}$ as our controlling parameter. Although this choice of definition may seem a bit artificial, it actually serves our goal better than the usual systems. The aim of the present paper is establishing a lower bound for any possible Frege system for some classes of logics and addressing a larger class of Frege systems with a possibly stronger parameter $\mathsf{M}$ is admittedly a stronger result. Moreover, later in the last section we will even use the mentioned strange system for $\mathsf{BPC}$ to provide a lower bound for the usual natural sequent-style proof system for $\mathsf{BPC}$. Therefore, investigating this larger class of systems is both strengthening and useful.\\

The following theorem constructs a Frege system for $\mathsf{BPC}$ with respect to $\mathsf{IPC}$.

\begin{thm}\label{ThmFregeForBPC}
There exists a Frege system $\mathbf{P}$ for $\mathsf{BPC}$ wrt $\mathsf{IPC}$ such that for multisets of formulas $\Upsilon=\gamma_1, \ldots, \gamma_m$, and $\Xi=\delta_1, \ldots , \delta_n$, if $\mathbf{BPC} \vdash^{\pi} \Upsilon \Rightarrow \Xi$ then
\[
\mathbf{P} \vdash^{\pi'} \bigwedge_{i=1}^m \gamma_i \to \bigvee_{j=1}^n \delta_j
\]
and $\lambda (\pi')$ is equal to the number of the sequents in $\pi$.
\end{thm}

\begin{proof}
Recall that for any sequent $S=\Upsilon \Rightarrow \Xi$ in the sequent calculus $\mathbf{BPC}$, the formula $I(S)$ is defined as $\bigwedge \Upsilon \to \bigvee \Xi$, where $\bigwedge \Upsilon= \top$ if $\Upsilon =\emptyset$ and $\bigvee \Xi = \bot$ if $\Xi=\emptyset$. Define a system $\mathbf{P}$ for $\mathsf{BPC}$ as the following. For any rule of the form 
\begin{center}
\begin{tabular}{c}
  \AxiomC{$T_1$}
   \AxiomC{$\ldots$}
   \AxiomC{$T_m$}
 \TrinaryInfC{$T$}
 \DisplayProof
\end{tabular} 
\end{center}
in $\mathbf{BPC}$, put $\Upsilon = \{\gamma\}$, $\Xi = \{\delta\}$, where $\gamma$ and $\delta$ are new formula variables, (except for the cut rule, where we put $\Upsilon_1 = \{A_1\}$, $\Upsilon_2 = \{A_2\}$, $\Xi_1 = \{B_1\}$, and $\Xi_2 \{B_2\}$, where $A_i$'s and $B_i$'s are new formula variables) and add the following rule to $\mathbf{P}$
\begin{center}
\begin{tabular}{c}
  \AxiomC{$I(T_1)$}
   \AxiomC{$\ldots$}
   \AxiomC{$I(T_m)$}
 \TrinaryInfC{$I(T)$}
 \DisplayProof
\end{tabular}
\end{center}
Note that this process also covers the initial sequents. Moreover, add the following two rules to $\mathbf{P}$:
\begin{center}
\begin{tabular}{c c}
  \AxiomC{$\phi$}
   \AxiomC{$\phi \to \psi$}
    \RightLabel{$(mp)$}
 \BinaryInfC{$\psi$}
 \DisplayProof
&
\hspace{2pt}
  \AxiomC{$\phi$}
   \AxiomC{$ \psi$}
    \RightLabel{$(adj)$}
 \BinaryInfC{$\phi \wedge \psi$}
 \DisplayProof
\end{tabular}
\end{center}
We will prove that $\mathbf{P}$ is a Frege system for $\mathsf{BPC}$ wrt $\mathsf{IPC}$. We have to check all the conditions of Definition \ref{DfnFrege}. First, it is an inference system with finitely many rules. Second, we have to show that $\mathbf{P}$ is sound, i.e., if $\mathbf{P} \vdash \phi$ then $\phi \in \mathsf{BPC}$. This can be proved using induction on the structure of the proof. As an example, suppose that the last rule used in the proof is 
\begin{center}
 \begin{tabular}{c}
  \AxiomC{$\Upsilon \Rightarrow (\phi \to \psi)$}
   \AxiomC{$\Upsilon \Rightarrow (\psi \to \theta)$}
   \RightLabel{$(Tr)$}
 \BinaryInfC{$\Upsilon \Rightarrow \Xi , (\phi \to \theta)$}
 \DisplayProof
\end{tabular}
\end{center}
whose corresponding Frege rule will be
\begin{center}
 \begin{tabular}{c}
  \AxiomC{$\gamma \to (\phi \to \psi)$}
   \AxiomC{$\gamma \to (\psi \to \theta)$}
 \BinaryInfC{$\gamma \to (\delta \vee (\phi \to \theta))$}
 \DisplayProof
\end{tabular}
\end{center}
By IH, $\gamma \to (\phi \to \psi) \in \mathsf{BPC}$ and $\gamma \to (\psi \to \theta) \in \mathsf{BPC}$. Therefore, $\mathbf{BPC} \vdash \, \Rightarrow \gamma \to (\phi \to \psi)$ and $\mathbf{BPC} \vdash \, \Rightarrow \gamma \to (\psi \to \theta)$. Since the cut elimination theorem holds in $\mathbf{BPC}$, we obtain  $\mathbf{BPC} \vdash \gamma \Rightarrow (\phi \to \psi)$ and $\mathbf{BPC} \vdash \gamma \Rightarrow (\psi \to \theta)$. Using the rule $(Tr)$ itself in $\mathbf{BPC}$, we get $\mathbf{BPC} \vdash (\gamma \Rightarrow \delta, \phi \to \theta)$ which implies $\mathbf{BPC} \vdash \, \Rightarrow \gamma \to (\delta \vee (\phi \to \theta))$, by the rules $(R \vee)$ and $(R \to)$. Hence $\gamma \to (\delta \vee (\phi \to \theta)) \in \mathsf{BPC}$. The cases for the other rules are similar.\\
Third, we have to show that $\mathbf{P}$ is strongly complete, i.e., if $\phi_1, \ldots, \phi_n \vdash_{\mathsf{BPC}} \phi$, then $\phi_1, \ldots, \phi_n \vdash_{\mathbf{P}} \phi$. It can be derived by showing the following: 
\begin{itemize}
\item[$(i)$]
If $\Upsilon \vdash_{\mathsf{BPC}} A$ then $\mathbf{BPC} \vdash \Upsilon \Rightarrow A$;
\item[$(ii)$]
if $\mathbf{BPC} \vdash \Upsilon \Rightarrow \Xi$ then $\mathbf{P} \vdash \bigwedge \Upsilon \to \bigvee \Xi$;
\item[$(iii)$]
$\phi_1, \ldots, \phi_n \vdash_{\mathbf{P}} \bigwedge_{i=1}^n \phi_i$.
\end{itemize}
The sketch of the proof for each part follows. 
\begin{itemize}
\item[$(i)$]
Observe that each rule in $\mathbf{BPC}$ has a context both in the antecedent of the premises and the conclusion. Therefore, for every rule of the form
\begin{center}
 \begin{tabular}{c}
  \AxiomC{$\Upsilon_1 \Rightarrow \Xi_1$}
   \AxiomC{$\ldots$}
   \AxiomC{$\Upsilon_n \Rightarrow \Xi_n$}
 \TrinaryInfC{$\Upsilon \Rightarrow \Xi$}
 \DisplayProof
\end{tabular}
\end{center}
the following is also a rule in $\mathbf{BPC}$, for any multiset $\Omega$
\begin{center}
 \begin{tabular}{c}
  \AxiomC{$\Omega, \Upsilon_1 \Rightarrow \Xi_1$}
   \AxiomC{$\ldots$}
   \AxiomC{$\Omega, \Upsilon_n \Rightarrow \Xi_n$}
 \TrinaryInfC{$\Omega, \Upsilon \Rightarrow \Xi$}
 \DisplayProof
\end{tabular}
\end{center}
It means that if we add a context $\Omega$ to the antecedents of all sequents in a proof, the result is also a proof in $\mathbf{BPC}$. Now, suppose $\Upsilon \vdash_{\mathsf{BPC}} A$ where $\Upsilon = \gamma_1, \ldots, \gamma_m$. Therefore, there exists a proof for $\Rightarrow A$ with $\Rightarrow \gamma_1, \ldots, \Rightarrow \gamma_m$ as initial sequents, in $\mathbf{BPC}$. Based on the observation, we can add $\Upsilon$ to the antecedent of each sequent in the proof and get a proof for $\Upsilon \Rightarrow A$ in $\mathbf{BPC}$ from the initial sequents $\Upsilon \Rightarrow \gamma_1, \ldots, \Upsilon \Rightarrow \gamma_m$. However, these initial sequents are instances of an initial sequent in $\mathbf{BPC}$ and hence we get $\vdash_{\mathbf{BPC}} \Upsilon \Rightarrow A$.
\item[$(ii)$]
It can be easily derived using induction on the structure of the proof. Suppose the premises of a rule are of the form $\Upsilon_1 \Rightarrow \Xi_1$ and $\Upsilon_2 \Rightarrow \Xi_2$ and the conclusion is $\Upsilon \Rightarrow \Xi$. Then by IH, we get $\mathbf{P} \vdash \bigwedge \Upsilon_1 \to \bigvee \Xi_1$ and $\mathbf{P} \vdash \bigwedge \Upsilon_2 \to \bigvee \Xi_2$. Using the corresponding rule in $\mathbf{P}$ we get $\mathbf{P} \vdash \bigwedge \Upsilon \to \bigvee \Xi$.
\item[$(iii)$]
This can be derived using the rule $(adj)$ in $\mathbf{P}$ for $n-1$ times.
\end{itemize}
Then, using these facts, we get $\phi_1, \ldots, \phi_n \vdash_{\mathbf{P}} \bigwedge_{i=1}^n \phi_i$ and $\vdash_\mathbf{P} \bigwedge_{i=1}^n \phi_i \to \phi$. Using the rule $(mp)$ we get $\phi_1, \ldots, \phi_n \vdash_{\mathbf{P}} \phi$.\\ %Note that if $\vdash_{\mathbf{BPC}} \Rightarrow \phi$ we obtain $\vdash_{\mathbf{P}} \top \to \phi$ and by modus ponens we get $\vdash_{\mathbf{P}} \phi$.\\
Finally, we have to show that each rule in $\mathbf{P}$ is $\mathsf{IPC}$-standard, which is a straightforward task. As an example, let us investigate the following rule of $\mathbf{P}$, which corresponds to the rule $(F \vee)$ in $\mathbf{BPC}$:
\begin{center}
 \begin{tabular}{c}
  \AxiomC{$\gamma \to (\phi \to \theta)$}
   \AxiomC{$\gamma \to (\psi \to \theta)$}
 \BinaryInfC{$\gamma \to (\delta \vee (\phi \vee \psi \to \theta))$}
 \DisplayProof
\end{tabular}
\end{center}
It is easy to see that 
\[
\Rightarrow \gamma \to (\phi \to \theta), \Rightarrow \gamma \to (\psi \to \theta) \vdash_{\mathbf{LJ}} \Rightarrow \gamma \to (\delta \vee (\phi \vee \psi \to \theta))
\]
which implies $\gamma \to (\phi \to \theta), \gamma \to (\psi \to \theta) \vdash_{\mathsf{IPC}} \gamma \to (\delta \vee (\phi \vee \psi \to \theta))$, i.e., the above rule is $\mathsf{IPC}$-standard. Similarly, all the other rules of $\mathbf{P}$ are $\mathsf{IPC}$-standard.\\
So far, we proved that $\mathbf{P}$ is a Frege system for $\mathsf{BPC}$ wrt $\mathsf{IPC}$. For the number of proof-lines, let $\pi= T_1, \ldots, T_n=(\Upsilon \Rightarrow \Xi)$ be a proof in $\mathbf{BPC}$, written in a way such that each $T_i$ is either an axiom or it is the conclusion of a rule with premises among the $T_j$'s where $j < i$. Then, it is enough to take $\pi'$ as $I(T_1), \ldots, I(T_n)=I(\Upsilon \Rightarrow \Xi)$. The new proof $\pi'$ is a proof in $\mathbf{P}$, since $\mathbf{P}$ is defined by imitating the rules and axioms of $\mathbf{BPC}$. It is also clear that $\lambda(\pi')$ is equal to the number of the sequents in $\pi$.
\end{proof}

\begin{dfn} \label{DfnEF}
An \emph{extended Frege} system for a substructural logic $\mathsf{L}$ is a Frege system for $\mathsf{L}$ together with the extension axiom, which allows formulas of the form $p \equiv \phi := (p \setminus \phi \wedge \phi \setminus p)$ to be added to a derivation with the following conditions: $p$ is a new variable not occurring in $\phi$, in any lines before $p \equiv \phi$, or in any hypotheses to the derivation. It can however appear in later lines, but not in the last line. An extended Frege system for a super-basic logic $\mathsf{L}$ wrt $\mathsf{M}$ is defined similarly, where $\mathsf{L} \subseteq \mathsf{M}$, by adding the extension axiom $p \equiv \phi := (p \to \phi \wedge \phi \to p)$ to a Frege system for $\mathsf{L}$ wrt $\mathsf{M}$.
\end{dfn}

It is easy to check that the definition of equivalence introduced in Definition \ref{DfnEF} is closed under substitution, i.e., for a substructural or a super-basic logic $\mathsf{L}$, if $A \equiv B$ holds in $\mathsf{L}$ then for any formula $\phi(p,\bar{q})$ we have $\phi(A, \bar{q}) \equiv \phi(B, \bar{q})$ in $\mathsf{L}$. %To prove this claim in a super-basic logic $\mathsf{L}$, for instance in the case that $\phi(p,q) = p \to q$, we make use of the fact that $\vdash_{\mathsf{L}} ((A \to B) \wedge (B \to C)) \to A \to C$.

\begin{lem} \label{LemEquivalenceOfEveryFrege}
For any two Frege systems $\mathbf{P}$ and $\mathbf{Q}$ for a substructural logic $\mathsf{L}$, there exists a number $c$ such that for any formula $\phi$ and any proof $\pi$, there exists a proof $\pi'$ such that
\[
\mathbf{P}\vdash^{\pi} \phi \;\;\; \textit{implies} \;\;\; \mathbf{Q} \vdash^{\pi'} \phi
\]
and $\lambda(\pi') \leq c \lambda (\pi)$. 
In the case that $\mathbf{P}$ and $\mathbf{Q}$ are extended Frege systems, they are polynomially equivalent.
\end{lem}

\begin{proof}
The proof is easy and originally shown in \cite{Cook}. Any instance of a rule in $\mathbf{P}$ can be replaced by its proof in $\mathbf{Q}$, which has a fixed number of lines. Take $c$ as the largest number of proof-lines of these proofs. Since there are finite many rules in $\mathbf{P}$, finding $c$ is possible. Therefore, $\lambda(\pi') \leq c \lambda (\pi)$. A similar argument also works for the lengths of the proofs when $\mathbf{P}$ and $\mathbf{Q}$ are extended Frege systems.
\end{proof}

As a result of Lemma \ref{LemEquivalenceOfEveryFrege}, since we are concerned with the number of proof-lines and lengths of proofs, we can talk about ``\emph{the}"  Frege (extended Frege) system for the substructural logic $\mathsf{L}$ and denote it by $\mathsf{L}-\mathbf{F}$ ($\mathsf{L}-\mathbf{EF}$). Note that Lemma \ref{LemEquivalenceOfEveryFrege} cannot be proved for any two Frege systems for $\mathsf{L}$ wrt $\mathsf{M}$ for super-basic logics $\mathsf{L} \subseteq \mathsf{M}$. For this to hold, we need an $\mathsf{L}-\mathbf{F}$ system wrt $\mathsf{M}$ to be strongly sound, i.e., if $\phi_1, \ldots, \phi_n \vdash_{\mathbf{P}} \phi$ then $\phi_1, \ldots, \phi_n \vdash_{\mathbf{L}} \phi$, which does not necessarily hold.

\begin{dfn}
A proof in a Frege (extended Frege, Hilbert-style, Gentzen-style) system is called \emph{tree-like} if every step of the proof is used at most once as a hypothesis of a rule in the proof. It is called a \emph{general} (or \emph{dag-like}) proof, otherwise.
\end{dfn}

In this paper we will not use this distinction, because throughout the paper all the proofs are considered to be dag-like, which is the more general notion. However, we find it useful to mention that the lower bound results of this paper actually work for any dag-like proof, including all the tree-like ones.

\section{A descent into the substructural world} \label{SectionDescent}
In this section, we will present a sequence of tautologies for the logic $\mathsf{FL} (\mathsf{BPC})$. Then, in Section \ref{SectionLowerBound} we show that they are exponentially hard for any proof system stronger than the sequent calculus $\mathbf{FL} (\mathbf{BPC})$ and polynomially weaker than a super-intuitionistic logic of infinite branching. This includes any extended Frege system for any substructural and super-basic logics weaker than a super-intuitionistic logic of infinite branching. For the super-basic case, the extended Frege system must be wrt the super-intuitionistic logic itself. In order to do so, we first provide some sentences provable in the weak system $\mathbf{WL}$. This uniformly provides two sequences of formulas provable in $\mathbf{FL}_{\bot}$ and $\mathbf{BPC}$. In the case of $\mathbf{FL}_{\bot}$, since the system $\mathbf{FL_{\bot}}$ is conservative over $\mathbf{FL}$ and the formulas we are interested in do not contain $\bot$, we will automatically have a proof in $\mathbf{FL}$.\\

To provide tautologies in $\mathbf{WL}$, we pursue the following strategy: First, using the representation $\{\bot, 1\}$ for true and false, we encode every binary evaluation of an $\mathbf{LK}$-formula by a suitable $\mathbf{WL}$-proof. Then, using this encoding, we map a certain fragment of $\mathbf{LK}$ into the system $\mathbf{WL}$, without any essential change in the original sequent. Finally, applying this map on certain hard intuitionistic tautologies provides the intended hard tautologies that we are looking for. 

\begin{dfn} \label{SubstitutionBotValuation}
Let $v$ be a Boolean valuation assigning truth values $\{t, f\}$ to the propositional variables. For a formula $A$ in the language $\mathcal{L}$, by $v(A)$ we mean the Boolean valuation of $A$ by $v$, defined in the usual way. The substitution $\sigma_v$ for a formula $A$ is defined in the following way: if the valuation $v$ assigns ``$t$" to an atom, substitute $1$ for this atom in $A$ and if $v$ assigns ``$f$" to an atom, then substitute $\bot$ for this atom in $A$. We write $A^{\sigma_v}$ for the formula obtained from this substitution.
\end{dfn}

\begin{lem}\label{EitherOr}
For any formula $A$ constructed from atoms and $\{\wedge, \vee\}$ and for any valuation $v$ we have
\[
\textit{if}\; \; v(A)=t, \; \textit{then}\; \mathbf{WL} \vdash A^{\sigma_v} \Leftrightarrow 1, \]
\[
 \textit{if}\; \; v(A)=f, \; \textit{then} \; \mathbf{WL} \vdash A^{\sigma_v} \Leftrightarrow \bot.
\]
\end{lem}

\begin{proof}
The proof is simple and uses induction on the structure of the formula $A$. If it is an atom, then the claim is clear by the definition of $A^{\sigma_v}$. If $A = B \wedge C$ then if $v(A)=t$, we have $v(B)=v(C)=t$. Therefore, by induction hypothesis we have
\[
\mathbf{WL} \vdash B^{\sigma_v} \Leftrightarrow 1 \;\; \textit{and} \;\;
\mathbf{WL} \vdash C^{\sigma_v} \Leftrightarrow 1
\]
Using the following proof-trees in $\mathbf{WL}$

\begin{center}
  	\begin{tabular}{c c}
	  	\AxiomC{$1 \Rightarrow B^{\sigma_v}$}
	  	\AxiomC{$1 \Rightarrow C^{\sigma_v}$}
	  	\RightLabel{$R \wedge$}
	  	\BinaryInfC{$1  \Rightarrow B^{\sigma_v} \wedge C^{\sigma_v}$}
  		\DisplayProof
  		&
  		\AxiomC{$B^{\sigma_v} \Rightarrow 1$}
	  	\RightLabel{$L \wedge_1$}
	  	\UnaryInfC{$ B^{\sigma_v} \wedge C^{\sigma_v} \Rightarrow 1$}
  		\DisplayProof
\end{tabular}
\end{center}
we obtain $\mathbf{WL} \vdash B^{\sigma_v} \wedge C^{\sigma_v} \Leftrightarrow 1$, which is $\mathbf{WL} \vdash A^{\sigma_v} \Leftrightarrow 1$.\\
If $A = B \wedge C$ and $v(A)=f$, then one of the following happens
\[
v(B)=t, v(C)=f \;\;\; \textit{or} \;\;\;
v(B)=f, v(C)=t \;\;\; \textit{or} \;\;\;
v(B)= v(C)=f 
\]
We investigate the first case, the other cases are similar. If $v(B)=t$ and $v(C)=f$, by induction hypothesis we get 
\[
\mathbf{WL} \vdash B^{\sigma_v} \Leftrightarrow 1 \;\; \textit{and} \;\;
\mathbf{WL} \vdash C^{\sigma_v} \Leftrightarrow \bot
\]
Therefore, the following are provable in $\mathbf{WL}$
\begin{center}
  	\begin{tabular}{c c c}
  		\AxiomC{$C^{\sigma_v} \Rightarrow \bot$}
	  	\RightLabel{$(L \wedge_2)$}
	  	\UnaryInfC{$ B^{\sigma_v} \wedge C^{\sigma_v} \Rightarrow \bot$}
  		\DisplayProof
  		&
  		\hspace{1pt}
  		&
  		\AxiomC{$\bot \Rightarrow B^{\sigma_v} \wedge C^{\sigma_v}$}
  		\DisplayProof
\end{tabular}
\end{center}
where the right sequent is an instance of the axiom for $\bot$. Hence, we get $\mathbf{WL} \vdash A^{\sigma_v} \Leftrightarrow \bot$.\\
Finally, if $A=B \vee C$, based on whether $v(A)=t$ or $v(A)=f$ we proceed as before. All the cases are simple, therefore here we only investigate the case where $v(A)=v(B \vee C)=t$ and $v(B)=f$ and $v(C)=t$, as an example. Using the induction hypothesis for $B$ and $C$, consider the following proof-trees in $\mathbf{WL}$:
\begin{center}
  	\begin{tabular}{c c }
  		\AxiomC{$1 \Rightarrow C^{\sigma_v}$}
	  	\RightLabel{$(R \vee_2)$}
	  	\UnaryInfC{$1 \Rightarrow B^{\sigma_v} \vee C^{\sigma_v}$}
  		\DisplayProof
  		&
  		\AxiomC{$B^{\sigma_v} \Rightarrow \bot$}
  		\AxiomC{$\bot \Rightarrow 1$}
	  	\RightLabel{$(cut)$}
	  	\BinaryInfC{$B^{\sigma_v} \Rightarrow 1$}
	  	\AxiomC{$C^{\sigma_v} \Rightarrow 1$}
	  	\RightLabel{$(L \vee)$}
	  	\BinaryInfC{$B^{\sigma_v} \vee C^{\sigma_v} \Rightarrow 1$}
  		\DisplayProof
\end{tabular}
\end{center}
\end{proof}

The following theorem is our main tool in proving the lower bound and it provides a method to convert classical tautologies to tautologies in $\mathbf{WL}$.

\begin{thm} \label{ClassicalTautFLe} 
Let $I = \{i_1, \ldots, i_k\} \subseteq \{1, \ldots, n\}$ and $A(\bar{p})$ is a formula only consisting of $\bar{p}= p_1, \ldots, p_n$ and connectives $\{\wedge, \vee\}$. If $ \bigwedge_{i_j \in I} p_{i_j} \to A(\bar{p})$ is a classical tautology, then we have 
\[
\mathbf{WL} \vdash \bigast_{j=1}^{k} (p_{i_j} \wedge 1) \Rightarrow A(\bar{p}).
\]
\end{thm}

\begin{proof}
First, note that in the theorem, due to the commutativity of conjunction in classical logic, any order on the elements of $I$, i.e., the sequence $i_1, \ldots, i_k$, can be used and $\bigast_{j=1}^{k} (p_{i_j} \wedge 1) \Rightarrow A(\bar{p})$ is provable in $\mathbf{WL}$. However, the order must be fixed throughout the proof. Moreover, we use $A$ and $A(\bar{p})$, interchangeably.
Now, for the proof, since $\bigwedge_{i_j \in I} p_{i_j} \to A(\bar{p})$ is a classical tautology, it will be true under any assignment of truth values to the propositional variables, especially the valuation $v$ assigning truth to every $p_{i_j}$, for $i_j \in I$, and falsity to the rest. It is easy to see that under this valuation we have $v(\bigwedge_{i_j \in I} p_{i_j})=t$ and since we also have $v(\bigwedge_{i_j \in I} p_{i_j} \to A(\bar{p}))=t$ (because the formula is a classical tautology), we get as a result that $v(A)=t$. Therefore, using Lemma \ref{EitherOr} we obtain
$
\mathbf{WL} \vdash A^{\sigma_v} \Leftrightarrow 1
$
and since $\mathbf{WL} \vdash \, \Rightarrow 1$, using the cut rule we get
\[
\mathbf{WL} \vdash \, \Rightarrow A^{\sigma_v} \;\;\;\;\; (\star) 
\]
On the other hand if we show
\[
\mathbf{WL} \vdash \bigast_{j=1}^{k} (p_{i_j} \wedge 1), A^{\sigma_v} \Rightarrow A \;\; (\dagger)
\]
then using the cut rule on the sequents in $(\star)$ and $(\dagger)$ we get
\[
\mathbf{WL} \vdash \bigast_{j=1}^{k} (p_{i_j} \wedge 1) \Rightarrow A.
\]

We will prove $(\dagger)$ by induction on the structure of the formula $A$. If $A$ is equal to $p_{i_j}$, for some $j$ where $i_j \in I$, then since $v(p_{i_j})=t$, we have $A^{\sigma_v}= p_{i_j}^{\sigma_v}= 1$. Therefore, the following proof-tree provides a proof in $\mathbf{WL}$:
\begin{center}
  	\begin{tabular}{c}
	  	\AxiomC{$p_{i_j} \Rightarrow p_{i_j}$}
	  	\RightLabel{$(L \wedge_1)$}
	  	\UnaryInfC{$p_{i_j} \wedge 1 \Rightarrow p_{i_j}$}
	  	\RightLabel{$(1 w)$}
	  	\UnaryInfC{$p_{i_j} \wedge 1, 1 \Rightarrow p_{i_j}$}
	  	\RightLabel{$(1 w)$}
	  	\UnaryInfC{$1, p_{i_j} \wedge 1, 1 \Rightarrow p_{i_j}$}
	  	\RightLabel{$(L \wedge_2)$}
	  	\UnaryInfC{$p_{i_{j-1}} \wedge 1, p_{i_j} \wedge 1, 1 \Rightarrow p_{i_j}$} 
	  	\UnaryInfC{$\strut \vdots$} 
	  	 \UnaryInfC{$p_{i_1} \wedge 1, \ldots, p_{i_{j-1}} \wedge 1, p_{i_j} \wedge 1, \ldots, p_{i_k} \wedge 1, 1 \Rightarrow p_{i_j}$}
	  	 \RightLabel{$(L *)$}
	  	 \UnaryInfC{$(p_{i_1} \wedge 1) * (p_{i_2} \wedge 1), \ldots, p_{i_{j-1}} \wedge 1, p_{i_j} \wedge 1, \ldots, p_{i_k} \wedge 1, 1 \Rightarrow p_{i_j}$}
	  	\UnaryInfC{$\strut \vdots$} 
	  	 \RightLabel{$(L *)$}
	  	\UnaryInfC{$ \bigast_{j=1}^{k} (p_{i_j} \wedge 1) , 1 \Rightarrow p_{i_j}$}
  		\DisplayProof
\end{tabular}
\end{center}
where the first vertical dots means applying the rules $(1 w)$ and $(L \wedge_2)$ consecutively for $k-2$ many times. Note that based on the rule $(1 w)$, we can add $1$ in any position on the left hand-side of the sequents. Using this fact together with the rule $(L \wedge_2)$, we obtain all formulas in the appropriate order. The second vertical dots represents applications of the rule $(L *)$ consecutively until one reaches the conclusion. Therefore, we have proved 
\[
\mathbf{WL} \vdash \bigast_{j=1}^{k} (p_{i_j} \wedge 1), A^{\sigma_v} \Rightarrow A
\]
for $A=p_{i_j}$ where $i_j \in I$.
The case where $A = p_{i_j}$ where $i_j \notin I$ is easier. Since for such $j$ we have $v(p_{i_j})=f$, using Lemma \ref{EitherOr} we get $ \mathbf{WL} \vdash A^{\sigma_v} \Leftrightarrow \bot$. Using the initial sequent for $\bot$ we have $ \mathbf{WL} \vdash \bigast_{j=1}^{k} (p_{i_j} \wedge 1), \bot \Rightarrow A$.\\

If $A(\bar{p}) = B(\bar{p}) \wedge C(\bar{p})$, by the induction hypothesis for $B(\bar{p})$ and $C(\bar{p})$,
\[
\mathbf{WL} \vdash \bigast_{j=1}^{k} (p_{i_j} \wedge 1), B^{\sigma_v} \Rightarrow B \;\; , \;\;
\mathbf{WL} \vdash \bigast_{j=1}^{k} (p_{i_j} \wedge 1), C^{\sigma_v} \Rightarrow C \;\; (\ddagger).
\]
Then, first using the rule $(L \wedge_1)$ for the left sequent and rule $(L \wedge_2)$ for the right sequent, and then using the rule $(R \wedge)$ we get 
\[
\mathbf{WL} \vdash \bigast_{j=1}^{k} (p_{i_j} \wedge 1), B^{\sigma_v} \wedge C^{\sigma_v} \Rightarrow B \wedge C.
\]
If $A(\bar{p}) = B(\bar{p}) \vee C(\bar{p})$, then by the induction hypothesis we have $(\ddagger)$. Then, first using the rule $(R \vee_1)$ for the left sequent in $(\ddagger)$ and the rule $(R \vee_2)$ for the right sequent in $(\ddagger)$, and then using $(L \vee)$ we get
\[
\mathbf{WL} \vdash \bigast_{j=1}^{k} (p_{i_j} \wedge 1), B^{\sigma_v} \vee C^{\sigma_v} \Rightarrow B \vee C.
\]
\end{proof}

\subsection{A brief digression into hard tautologies} \label{SubsectionDigression}
The formulas we are going to introduce as our hard tautologies for the system $\mathsf{FL}-\mathbf{EF}$ and any $\mathsf{BPC}-\mathbf{EF}$ wrt a super-intuitionistic logic of infinite branching $\mathsf{M}$ are inspired by the hard formulas for $\mathsf{IPC}-\mathbf{F}$ introduced by Hrube\v{s} \cite{Hrubes} and their negation-free version introduced by Je\v{r}\'{a}bek \cite{Jerabek}. In this subsection, we present these formulas and what combinatorial facts they represent.\\

Let us first define formulas $Clique_{n,k}$ and $Color_{n,m}$ which will be used in Hrube\v{s}'s formulas.

\begin{dfn} \cite[Section 13.5]{KrajicekProof}
Let $n, k, m \geq 1$. By an undirected simple graph on $[n]$ we mean the set of strings of length $\binom{n}{2}$, representing the edge data of a graph over $n$ vertices. We will use the string and the graph it represents, interchangeably. We say a graph has a clique of size at least $k$, when it has a complete subgraph of size $k$, i.e., a subgraph containing all possible edges among its vertices. Define $Clique_{n,k}$ to be the set of undirected simple graphs on $[n]$ that have a clique of size at least $k$. Define $Color_{n,m}$ to be the set of garphs on $[n]$ that are $m$-colorable, i.e., it is possible to color its vertices by $m$ colors such that no connected vertices have the same color. The formulas representing the sets $Clique_{n,k}$ and $Color_{n,m}$ are defined as follows.\\
Consider the following set of clauses using $\binom{n}{2}$ atoms $p_{ij}$, $\{i,j\} \in \binom{[n]}{2}$, one for each potential edge in a graph on $[n]$, and $k.n$ atoms $q_{ui}$ intended to describe a mapping from $[k]$ to $[n]$:
\begin{itemize}
\item[$\bullet$] 
$\bigvee_{i \in [n]} q_{ui}$, for all $u \leq k$,
\item[$\bullet$] 
$\neg q_{ui} \vee \neg q_{uj}$, for all $u \in [k]$ and any $i \neq j \in [n]$,
\item[$\bullet$] 
$\neg q_{ui} \vee \neg q_{vi}$, for all $u \neq v \in [k]$ and any $i \in [n]$,
\item[$\bullet$] 
$\neg q_{ui} \vee \neg q_{vj} \vee p_{ij}$, for all $u \neq v \in [k]$ and any $\{i,j\} \in \binom{[n]}{2}$.
\end{itemize}
Define $Clique^k_n(\bar{p}, \bar{q})$ as the conjunction of all of these formulas. Now, consider the following set of clauses using atoms $\bar{p}$ and $n.m$ more atoms $r_{ia}$ where $i \in [n]$ and $a \in [m]$, intended to describe an $m$-coloring of the graph as a function from $[n]$ to $[m]$:
\begin{itemize}
\item[$\bullet$] 
$\bigvee_{a \in [m]} r_{ia}$, for all $i \in [n]$,
\item[$\bullet$] 
$\neg r_{ia} \vee \neg r_{ib}$, for all $a \neq b \in [m]$ and any $i \in [n]$,
\item[$\bullet$] 
$\neg r_{ia} \vee \neg r_{ja} \vee \neg p_{ij}$, for all $a \in [m]$ and any $\{i,j\} \in \binom{[n]}{2}$.
\end{itemize}
Define $Color_{n}^m(\bar{p}, \bar{r})$ as the conjunction of all of these formulas.
Note that every occurrence of atoms $p_{ij}$ in $Clique^k_n(\bar{p}, \bar{q})$ is positive, or in other words it is monotone in $\bar{p}$.
\end{dfn}

The exponential lower bound for intuitionistic logic is demonstrated in the following theorem due to P. Hrube\v{s}. The main idea is that any short proof for the hard tautology provides a small monotone circuit to decide whether a given graph is not $k$-colorable or it does not have a $(k+1)$-clique, which we know is a hard problem to decide \cite{Alon}.

\begin{thm}\label{Hrubes} \cite{Hrubes}
Let $\bar{p}=p_1, \cdots, p_n$ and $\bar{q}=q_1, \cdots, q_n$ and $\bar{p}, \bar{q}, \bar{r}, \bar{s}$ be disjoint variables, %$\bar{v}=\{\bar{p}, \bar{q}, \bar{r}, \bar{s}\}$, 
and $k= \lfloor \sqrt{n} \rfloor$. Then the formulas
\[
\Theta^{\bot}_n:= \bigwedge_{i=1, \cdots, n} (p_i \vee q_i) \to \neg Color^{k}_n(\bar{p}, \bar{s}) \vee \neg Clique^{k+1}_n (\neg \bar{q}, \bar{r})
\]
are intuitionistic tautologies. Moreover, every $\mathsf{IPC}-\mathbf{F}$-proof of $\Theta^{ \bot}_n$ contains at least $2^{\Omega(n^{1/4})}$ proof-lines.
\end{thm}

We refer to the formulas $\Theta^{\bot}_n$ as Hrube\v{s}'s formulas. The superscript $\bot$ in $\Theta^{\bot}_n$ stresses that the formulas contain negations.
For our purposes, we need to use a negation-free version of Hrube\v{s}'s formulas.

\begin{dfn} \cite[Definition 6.28]{Jerabek} \label{Jerabek}
For $k \leq n$ define:
\[
\alpha^k_n (\bar{p}, \bar{s}, \bar{s'}):= \bigvee_{i < n} \bigwedge_{l <k} s'_{i, l} \vee \bigvee_{i,j <n} \bigvee_{l <k} (s_{i,l} \wedge s_{j,l} \wedge p_{i,j}),
\]
\[
\beta^k_n (\bar{q}, \bar{r}, \bar{r'}) := \bigvee_{l< k} \bigwedge_{i <n} r'_{i, l} \vee \bigvee_{i,j <n} \bigvee_{l<m <k} (r_{i,l} \wedge r_{j,m} \wedge q_{i,j}).
\]
Define the negation-free version of Hrube\v{s}'s formulas as follows:
\[
\Theta_{n,k} := \bigwedge_{i,j} (p_{i,j} \vee q_{i,j}) \to [(\bigwedge_{i,l}(s_{i,l} \vee s'_{i,l}) \to \alpha^k_n (\bar{p}, \bar{s}, \bar{s'})) \vee (\bigwedge_{i,l}(r_{i,l} \vee r'_{i,l}) \to \beta^{k+1}_n (\bar{q}, \bar{r}, \bar{r'}))].
\]
Notice that $Color^k_n (\bar{p}, \bar{s}) = \neg \alpha^k_n (\bar{p}, \bar{s}, \neg \bar{s})$ and $Clique^k_n (\bar{p}, \bar{r})= \neg \beta^k_n (\neg \bar{p}, \bar{r}, \neg \bar{r})$. Denote $\Theta_{n, \lfloor \sqrt{n} \rfloor}$ by $\Theta_{n}$ for simplicity. The lower bound of Theorem \ref{Hrubes} also applies to $\Theta_n$ \cite{Jerabek}. 
\end{dfn}

What Je\v{r}\'{a}bek did to make Hrube\v{s}'s formulas negation-free was to introduce new propositional variables $s'_{i,l}$ and $r'_{i,l}$ to play the role of $\neg s_{i,l}$ and $\neg r_{i,l}$, respectively. This trick provides some implication-free formulas $\alpha^k_n$ and $\beta^k_n$ in the definition \ref{Jerabek} which also makes the formulas $\Theta_n $ more amenable to the technique that we provided in Section \ref{SectionDescent}. 

\begin{thm}(\cite[Theorem 6.37]{Jerabek}) \label{JerabekAsli}
Let $\mathsf{L}$ be a super-intuitionistic logic of infinite branching. Then the formulas $\Theta_n$ are intuitionistic tautologies and they require $\mathsf{L}-\mathbf{EF}$-proofs of length $2^{n^{\Omega(1)}}$, and $\mathsf{L}-\mathbf{F}$-proofs with at least $2^{n^{\Omega(1)}}$ lines.
\end{thm}

\subsection{Weak hard tautologies}
The following lemmas are easy observations. The first one states that fusion distributes over disjunction in substructural logics. The second one presents a property of the sequent calculus $\mathbf{LK}$.

\begin{lem} \label{Distributivity}
In the sequent calculus $\mathbf{WL}$ we have the following:
\[
\mathbf{WL} \vdash \bigast_{i=1}^n (A_i \vee B_i) \Leftrightarrow \bigvee_{I} (\bigast_{i=1}^n D_{i}^{I})
\]
where $I \subseteq \{1, 2, \cdots, n\}$ and $D_{i}^{I} =\left\{
                \begin{array}{ll}
                 A_i  \;\; , \;\; i \in I\\
                  B_i \;\; , \;\; i \notin I
                \end{array}
              \right.$.
\end{lem}

\begin{proof}
The proof is easy and uses induction on $n$. Note that in each disjunct in the right hand-side, $D_i^{I}$ is either $A_i$ or $B_i$, according to the subset $I$. However, the order of the subscripts must be increasing. For instance, for the case $n=2$ we have 
\[
\mathbf{WL} \vdash (A_1 \vee B_1) * (A_2 \vee B_2) \Leftrightarrow (A_1 * B_2) \vee (A_1 * A_2) \vee (B_1 * A_2) \vee (B_1 * B_2).
\]
\end{proof}

\begin{lem} \label{LemmaMVP}
Suppose $\alpha_1 \to \alpha_2$ and $\beta_1 \to \beta_2$ have no propositional variables in common.
If the formula $\alpha_1 \wedge \beta_1 \to \alpha_2 \vee \beta_2$ is provable in $\mathbf{LK}$, then either $\alpha_1 \to \alpha_2$ or $\beta_1 \to \beta_2$ is provable in $\mathbf{LK}$.
\end{lem}

\begin{proof}
It is an easy corollary of Craig's interpolation theorem, applied to the provable formula $\alpha_1 \wedge \neg \alpha_2 \to (\beta_1 \to \beta_2)$. 
\end{proof}

We are now ready to formulate the hard tautologies and prove the lower bound. Note that in the formula $\bigast_{i=1}^{n-1} \bigast_{j=1}^{n-1} A_{i,j}$, the indices first range over $j$ and then over $i$, which will result in the lexicographic order, i.e., it has the following form
\[
A_{1,1} * A_{1,2} * \cdots * A_{1, n-1} * A_{2,1} * \cdots * A_{n-1, n-1}.
\]
For a set of formulas $X$, by $\parallel X \parallel$ we mean the number of elements of the set. %(the number of formulas the sequence contains).

\begin{thm} \label{HardTautologiesFLe}
The formulas
\[
\Theta^{*}_{n,k}:= [ \bigast_{i=1}^{n-1} \bigast_{j=1}^{n-1} ((p_{i,j} \wedge 1) \vee (q_{i,j} \wedge 1))] \; \; \setminus
\]
\[
[\bigast_{i=1}^{n-1} \bigast_{l=1}^{k-1} ((s_{i,l} \wedge 1) \vee (s'_{i,l} \wedge 1)) \setminus \alpha^k_n (\bar{p}, \bar{s}, \bar{s'})] \; \vee \; [\bigast_{i=1}^{n-1} \bigast_{l=1}^{k-1}((r_{i,l} \wedge 1) \vee (r'_{i,l} \wedge 1)) \setminus \beta^{k+1}_n (\bar{q}, \bar{r}, \bar{r'})]
\]
are provable in $\mathbf{WL}$ for all $1 \leq k \leq n$. 
\end{thm}

\begin{proof} 
Let us denote the following formula, which is the right-hand side of $\setminus$, by $A$:
\[
[\bigast_{i=1}^{n-1} \bigast_{l=1}^{k-1} ((s_{i,l} \wedge 1) \vee (s'_{i,l} \wedge 1)) \setminus \alpha^k_n (\bar{p}, \bar{s}, \bar{s'})] \; \vee \; [\bigast_{i=1}^{n-1} \bigast_{l=1}^{k-1}((r_{i,l} \wedge 1) \vee (r'_{i,l} \wedge 1)) \setminus \beta^{k+1}_n (\bar{q}, \bar{r}, \bar{r'})].
\]
First, we show 
\[
\mathbf{WL} \vdash \bigast_{i=1}^{n-1} \bigast_{j=1}^{n-1} Q_{i,j}^{I} \Rightarrow A \;\;\;\;\;\; (\dagger)
\]
where $Q_{i,j}=\left\{
                \begin{array}{ll}
                 p_{i,j} \wedge 1 \;\; , \;\; (i,j) \in I\\
                 q_{i,j} \wedge 1  \;\; , \;\; (i,j) \notin I 
                \end{array}
              \right.$
for any $ I \subseteq \{ (i,j) \mid i,j \in \{1, \cdots, n-1\} \}$.\\
For simplicity from now on, unless specified otherwise, we will delete the ranges of $i$, $j$ and $l$, which are indicated in $\Theta^{*}_{n,k}$.\\
It is easy to see how proving $(\dagger)$ will result in proving the theorem. The reason is the following. Since $(\dagger)$ is provable for any $I \subseteq \{ (i,j) \mid i,j \in \{1, \cdots, n-1\} \}$, using the left disjunction rule for $2^{(n-1)^2}-1$ many times on $(\dagger)$, we get
\[
\mathbf{WL} \vdash \bigvee_{I} \bigast_i \bigast_j Q_{i,j}^{I} \Rightarrow A.
\]
Furthermore, Lemma \ref{Distributivity} allows us to obtain 
\[
\mathbf{WL} \vdash \bigast_{i} \bigast_{j} ((p_{i,j} \wedge 1) \vee (q_{i,j} \wedge 1))
\Rightarrow
\bigvee_{I} \bigast_i \bigast_j Q_{i,j}^{I},
\]
and using the cut rule and the rule $(R \setminus)$, we conclude
\[
\mathbf{WL} \vdash \, \Rightarrow \Theta^{*}_{n,k} .
\] 
On the other hand, $\Theta_{n,k}$ presented in Definition \ref{Jerabek}, is provable in $\mathbf{LJ}$, and therefore also provable in $\mathbf{LK}$. Using the distributivity of conjunction over disjunction we have 
\[
\mathbf{LK} \vdash \bigwedge_{(i,j) \in M} p_{i,j} \wedge \bigwedge_{(i,j) \in N} q_{i,j} \Rightarrow [\bigwedge_{i,l}(s_{i,l} \vee s'_{i,l}) \to \alpha^k_n (\bar{p}, \bar{s}, \bar{s'})] \vee [\bigwedge_{i,l}(r_{i,l} \vee r'_{i,l}) \to \beta^{k+1}_n (\bar{q}, \bar{r}, \bar{r'})]
\]
for any $M$ and $N$as a partition for the set $\{ (i,j) \mid i,j \in \{1, \cdots, n-1\} \}$.
For such $M$ and $N$, using Lemma \ref{LemmaMVP} we have either 
\[
\mathbf{LK} \vdash \bigwedge_{(i,j) \in M} p_{i,j} \Rightarrow (\bigwedge_{i,l}(s_{i,l} \vee s'_{i,l}) \to \alpha^k_n (\bar{p}, \bar{s}, \bar{s'})),  
\]
or
\[
\mathbf{LK} \vdash \bigwedge_{(i,j) \in N} q_{i,j} \Rightarrow (\bigwedge_{i,l}(r_{i,l} \vee r'_{i,l}) \to \beta^{k+1}_n (\bar{q}, \bar{r}, \bar{r'})).
\]
We consider the first case. The second one is similar.
Suppose the first case holds. Using the cut rule, we have 
\[
\mathbf{LK} \vdash \bigwedge_{(i,j) \in M} p_{i,j} \; , \; \bigwedge_{i,l}(s_{i,l} \vee s'_{i,l}) \Rightarrow  \alpha^k_n (\bar{p}, \bar{s}, \bar{s'}),
\]
and using the left exchange rule we obtain
\[
\mathbf{LK} \vdash \bigwedge_{i,l}(s_{i,l} \vee s'_{i,l}) \; , \; \bigwedge_{(i,j) \in M} p_{i,j} \Rightarrow  \alpha^k_n (\bar{p}, \bar{s}, \bar{s'}).
\]
Now, using the distributivity of conjunction over disjunction in $\mathbf{LK}$, for any $U$ and $V$ as a partition for the set $\{(i,l) \mid i < n, l < k\}$, we have
\[
\mathbf{LK} \vdash (\bigwedge_{(i,l) \in U} s_{i,l} \wedge \bigwedge_{(i,l) \in V} s'_{i,l}) \; , \; \bigwedge_{(i,j) \in M} p_{i,j} \Rightarrow  \alpha^k_n (\bar{p}, \bar{s}, \bar{s'}) , 
\]
or equivalently (using the rules $(L \wedge_1), (L \wedge_2)$, and left contraction),
\[
\mathbf{LK} \vdash  \bigwedge_{(i,l) \in U} s_{i,l} \wedge \bigwedge_{(i,l) \in V} s'_{i,l} \wedge \bigwedge_{(i,j) \in M} p_{i,j} \Rightarrow  \alpha^k_n (\bar{p}, \bar{s}, \bar{s'}).   
\]
Now, using Theorem \ref{ClassicalTautFLe}, we have 
\[
\mathbf{WL} \vdash  (\bigast_{i=1}^{n-1} \bigast_{l=1}^{k-1} S_{i,l}^{U, V}) \; * \; (\bigast_{(i,j) \in M} (p_{i,j} \wedge 1)) \Rightarrow  \alpha^k_n (\bar{p}, \bar{s}, \bar{s'}),  
\]
where $S_{i,l}^{U,V}=\left\{
                \begin{array}{ll}
                 s_{i,l} \wedge 1  \;\; , \;\; (i,j) \in U\\
                 s'_{i,l} \wedge 1 \;\; , \;\; (i,j) \in V 
                \end{array}
              \right.$.\\
Note that by Theorem \ref{ClassicalTautFLe}, we can choose any order on $\bigast_{(i,j) \in M} (p_{i,j} \wedge 1)$. Here, we fix the lexicographic order on $\{(i,j) \mid i,j \in \{1, \ldots, n-1\}\}$.\\
Equivalently (using the fact that for any formulas $A$ and $B$, we have $\mathbf{WL} \vdash A, B \Rightarrow A *B$ and then using the cut rule), we have
\[
\mathbf{WL} \vdash (\bigast_{i=1}^{n-1} \bigast_{l=1}^{k-1} S_{i,l}^{U, V}) \; , \; (\bigast_{(i,j) \in M} (p_{i,j} \wedge 1)) \Rightarrow  \alpha^k_n (\bar{p}, \bar{s}, \bar{s'}).  
\]
Since this sequent is provable for any $U$ and $V$ as a partition for $\{(i,l) \mid i < n, l < k\}$, using the left disjunction rule for $2^{(n-1)(k-1)}-1$ many times we get
\[
\mathbf{WL} \vdash
\bigvee_{U, V} (\bigast_{i=1}^{n-1} \bigast_{l=1}^{k-1} S_{i,l}^{U, V}) \; , \; (\bigast_{(i,j) \in M} (p_{i,j} \wedge 1)) \Rightarrow  \alpha^k_n (\bar{p}, \bar{s}, \bar{s'}).
\]
Using Theorem \ref{Distributivity} and the cut rule we have
\[
\mathbf{WL} \vdash [\bigast_{i} \bigast_l ((s_{i,l} \wedge 1) \vee (s'_{i,l} \wedge 1))] \; , \; (\bigast_{(i,j) \in M} (p_{i,j} \wedge 1)) \Rightarrow  \alpha^k_n (\bar{p}, \bar{s}, \bar{s'}),  
\]
and using the rule $(R \setminus)$ we get 
\[
\mathbf{WL} \vdash \bigast_{(i,j) \in M} (p_{i,j} \wedge 1) \Rightarrow [\bigast_{i} \bigast_l ((s_{i,l} \wedge 1) \vee (s'_{i,l} \wedge 1))] \; \setminus \; \alpha^k_n (\bar{p}, \bar{s}, \bar{s'}).
\]
Recall that the order on $\bigast_{(i,j) \in M}(p_{i,j} \wedge 1)$ is the lexicographic order. Now, by using the rules $(L 1)$ and $(L \wedge_2)$ consecutively for $\parallel N \parallel$-many times and whenever needed (each time producing $q_{i,j} \wedge 1$ for each element of $N$, in the same manner as in the proof of Theorem \ref{ClassicalTautFLe} and in the right place) and then using the rule $(L *)$ for $\parallel N \parallel$-many times and in the end using the rule $(R \vee_1)$, we prove $(\dagger)$.
\end{proof}

\begin{rem}
It is worth noting that the system $\mathbf{WL}$ could have been defined in an alternative way by deleting $/$ instead of $\setminus$ from the language, and having the same initial sequents and rules as $\mathbf{FL_{\bot}}$ and leaving the rules $(R /), (L /),$ and $(L \setminus)$ out. Then, in a similar manner, the following formulas would be provable in this alternative calculus:
\[
[\alpha^k_n (\bar{p}, \bar{s}, \bar{s'}) / \bigast_{i=1}^{n-1} \bigast_{l=1}^{k-1} ((s_{i,l} \wedge 1) \vee (s'_{i,l} \wedge 1))] \; \vee \;
 [\beta^{k+1}_n (\bar{q}, \bar{r}, \bar{r'}) / \bigast_{i=1}^{n-1} \bigast_{l=1}^{k-1}((r_{i,l} \wedge 1) \vee (r'_{i,l} \wedge 1))] \;\; / 
\]
\[
[\bigast_{i=1}^{n-1} \bigast_{j=1}^{n-1} ((p_{i,j} \wedge 1) \vee (q_{i,j} \wedge 1)].
\]
\end{rem}

Now, we are ready to present tautologies in $\mathsf{FL}$ and $\mathsf{BPC}$. It is easy to see that the tautologies introduced in Theorem \ref{HardTautologiesFLe} are provable in basic substructural logics.

\begin{cor} \label{CorThetaFL}
The formulas $\Theta^{*}_{n,k}$ are provable in the logic $\mathsf{FL}$. 
\end{cor}

\begin{proof}
Clearly, $\mathbf{WL}$ is a subsystem of the sequent calculus $\mathbf{FL_{\bot}}$. Then, using the cut elimination theorem for $\mathbf{FL_{\bot}}$ \cite{Ono}, and the fact that $\Theta^{*}_{n,k}$ do not contain $\bot$, we obtain the result. 
\end{proof}

To provide tautologies in $\mathsf{BPC}$, we need the translation function $t$, defined in Section \ref{SectionFregeExtendedFrege}.

\begin{cor} \label{CorBPCFL} 
The formulas $(\Theta^{*}_{n,k})^t$ are provable in $\mathsf{BPC}$.
\end{cor}

\begin{proof}
The provability of $(\Theta^{*}_{n,k})^t$ is a consequence of Theorem \ref{ClassicalTautFLe} and Lemma \ref{LemTranslationWLBPC}. 
\end{proof}

From now on, since we only use the case where $k= \lfloor \sqrt{n} \rfloor$ and we will prove the lower bound for this case, we fix $k = \lfloor \sqrt{n} \rfloor$ and use the notation $\Theta^{*}_n = \Theta^{*}_{n, \lfloor \sqrt{n} \rfloor}$ for it.

\section{The main theorem} \label{SectionLowerBound} 
In this section we will present the main result of the paper. We will prove that there exists an exponential lower bound on the lengths of proofs in proof systems for a wide range of logics. Furthermore, we will obtain an exponential lower bound on the number of proof-lines in a broad range of Frege systems.

\begin{thm} \label{MainTheorem} 
Let $\mathsf{L}$ be a super-intitionistic logic of infinite branching, which has a Frege system.
\begin{itemize}
\item[$(i)$]
Let $\mathbf{P}$ be a proof system for a logic with the language $\mathcal{L}^*$ such that $\mathbf{P}$ is at least as strong as $\mathbf{FL}$ and $\mathbf{P} \leq^{t} \mathsf{L}-\mathbf{EF}$. Then $\mathbf{P} \vdash \Theta^*_n$ and the length of any such proof is exponential in $n$. 
\item[$(ii)$]
Let $\mathbf{P}$ be a proof system for a logic with the language $\mathcal{L}$ such that $\mathbf{P}$ is at least as strong as $\mathbf{BPC}$ and $\mathbf{P} \leq \mathsf{L}-\mathbf{EF}$. Then $\mathbf{P} \vdash (\Theta^*_n)^t$ and the length of any such proof is exponential in $n$. 
\end{itemize}
\end{thm}

\begin{proof}
$(i)$ First, observe that by Lemma \ref{LemEquivalenceOfEveryFrege}, any two extended Frege systems for $\mathsf{L}$ are polynomially equivalent. Hence, w.l.o.g. we will assume that $\mathsf{L}-\mathbf{EF}$ has an explicit modus ponens rule. Then, note that since $\mathbf{FL} \vdash \Theta^*_n$, by Theorem \ref{HardTautologiesFLe}, and the assumption that $\mathbf{P}$ is at least as strong as $\mathbf{FL}$, the formulas $\Theta^*_n$ are also provable in $\mathbf{P}$.
Take such a proof $\pi$, i.e., $\mathbf{P} \vdash^{\pi} \Theta^*_n$. Since $\mathbf{P} \leq^{t} \mathsf{L}-\mathbf{EF}$, there exists a polynomial $p$ and a proof $\pi'$, such that $\mathsf{L}-\mathbf{EF} \vdash^{\pi'} (\Theta^*_n)^t$ and $|\pi'| \leq p (|\pi|)$. As another step towards the claim, we want to provide $x$ such that $\mathsf{L}-\mathbf{EF} \vdash^x (\Theta^*_n)^t \to \Theta_n$ with its length, $|x|$, and therefore the number of its  proof-lines, $\lambda(x)$, is polynomial in $n$. Note that the formulas $\Theta_n$ and $(\Theta^*_n)^t$ are almost identical, except that for any atom $u$ in $\Theta_n$ the formula $u \wedge \top$ is present in its place in $(\Theta^*_n)^t$. Now, since $\mathsf{L}$ is a super-intuitionistic logic, we have $\mathsf{L}-\mathbf{EF} \vdash u \wedge \top \leftrightarrow u$, for any atom $u$ present in the formula $\Theta_n$. This proof has a fix number of proof-lines in $\mathsf{L}-\mathbf{EF}$. The claim then easily follows from the fact that the length of the formula $\Theta_n$, hence the number of the connectives and the number of the atoms $u$ in $\Theta_n$ are also polynomial in $n$. It means that $|x|$ and hence $\lambda(x)$ is polynomial in $n$. Therefore, putting $\pi'$ and $x$ together, since we have the explicit modus ponens rule in $\mathsf{L}-\mathbf{EF}$, we can show that the formula $\Theta_n$ is provable in $\mathsf{L}-\mathbf{EF}$ with a proof polynomially long in $n$ and $|\pi'|$. (This indicates that the number of the proof-lines of $\Theta_n$ in $\mathsf{L}-\mathbf{EF}$ is also polynomial in $|\pi'|$ and $n$. We will use this fact in the proof of Theorem \ref{ThmFLBPC-EF}) By Theorem \ref{JerabekAsli}, any $\mathsf{L}-\mathbf{EF}$-proof of $\Theta_n$ has length at least $2^{\Omega(n^{1/4})}$. Therefore, the length of $\pi'$, and hence the length of $\pi$, must be exponential in $n$.\\
$(ii)$ The proof for this part is similar to that of $(i)$. Here, by Corollary \ref{CorBPCFL}, the formulas $(\Theta^*_n)^t$ are provable in $\mathbf{BPC}$ and hence in $\mathbf{P}$. Since $\mathsf{L}-\mathbf{EF}$ polynomially simulates $\mathbf{P}$, we obtain the exponential lower bound using the fact that $\mathsf{L}-\mathbf{EF} \vdash (\Theta^*_n)^t \to \Theta_n$ with a proof polynomially long in $n$.
\end{proof}

The following theorem states an exponential lower bound on the number of proof-lines in a wide range of Frege systems.

\begin{thm} \label{ThmFLBPC-EF}
Let $\mathsf{M}$ be a super-intitionistic logic of infinite branching, which has a Frege system.
\begin{itemize}
\item[$(i)$]
Let $\mathsf{L}$ be a substructural logic with the language $\mathcal{L^*}$ such that $\mathsf{L} \subseteq^t \mathsf{M}$. Then, the number of lines of every proof of $\Theta^*_n$ in $\mathsf{L}-\mathbf{F}$ is exponential in $n$ and every proof of $\Theta^*_n$ in $\mathsf{L}-\mathbf{EF}$ has length exponential in $n$.
\item[$(ii)$]
Let $\mathsf{L}$ be a super-basic logic with the language $\mathcal{L}$ such that $\mathsf{L} \subseteq \mathsf{M}$. Then, the number of lines of every proof of $(\Theta^*_n)^t$ in any $\mathsf{L}-\mathbf{F}$ system wrt $\mathsf{M}$ is exponential in $n$ and every proof of $(\Theta^*_n)^t$ in any $\mathsf{L}-\mathbf{EF}$ system wrt $\mathsf{M}$ has length exponential in $n$. 
\end{itemize}
\end{thm}

\begin{proof}
In order to use Theorem \ref{MainTheorem}, we have to show that for any given $\mathsf{L}-\mathbf{EF}$, it is as strong as $\mathbf{FL}$ and also we have $\mathsf{L}-\mathbf{EF} \leq^t \mathsf{M}-\mathbf{EF}$. Since $\mathsf{L}$ is a substructural logic, clearly $\mathsf{L}-\mathbf{EF}$ is as strong as $\mathbf{FL}$. For the second condition to hold, we will provide an extended Frege system $\mathbf{P}$ for $\mathsf{M}$ such that for any proof $\pi= \phi_1, \ldots, \phi_m$ in $\mathsf{L}-\mathbf{EF}$, $\pi^t= \phi_1^t, \ldots, \phi_m^t$ will be a proof in $\mathbf{P}$. Clearly, the function $g$ in Definition \ref{DfnProofSystemSimulate} is the identity function since $|\pi|=|\pi^t|$. This is because for any formula $\phi$ in the language $\mathcal{L^*}$ we have $|\phi|= |\phi^t|$, therefore the length of $\pi$ is the same as the length of $\pi^t$ and also $\lambda(\pi) = \lambda(\pi^t)$.\\
Fix an extended Frege system $\mathbf{Q}$ for the logic $\mathsf{M}$. This is possible by the assumption of the existence of a Frege system for $\mathsf{M}$. Define the system $\mathbf{P}$ as the system consisting of all the rules in $\mathbf{Q}$, the modus ponens rule, and the rules:
\begin{center}
  	\begin{tabular}{c}
  		\AxiomC{$A_1^t$}
  		\AxiomC{$\ldots$}
	  	\AxiomC{$A_l^t$}
	  	\TrinaryInfC{$A^t$}
  		\DisplayProof
\end{tabular}
\end{center}
for any rule of $\mathsf{L}-\mathbf{EF}$ of the form:
\begin{center}
  	\begin{tabular}{c}
	  	\AxiomC{$A_1$}
	  	\AxiomC{$\ldots$}
	  	\AxiomC{$A_l$}
	  	\TrinaryInfC{$A$}
  		\DisplayProof
\end{tabular}
\end{center}
First of all, we have to show that $\mathbf{P}$ is an extended Frege system for the logic $\mathsf{M}$. Therefore, we have to check that $\mathbf{P}$ satisfies all the conditions in Definition \ref{DfnFrege}. Condition $1$ is obvious and condition $2$ is a result of condition $4$. $\mathbf{P}$ is strongly complete for $\mathsf{M}$, since it contains $\mathbf{Q}$ and $\mathbf{Q}$ is strongly complete for $\mathsf{M}$, therefore condition $3$ holds. For condition $4$, note that for any rule in $\mathsf{L}-\mathbf{EF}$ of the form:
\begin{center}
  	\begin{tabular}{c}
	  	\AxiomC{$A_1$}
	  	\AxiomC{$\ldots$}
	  	\AxiomC{$A_l$}
	  	\TrinaryInfC{$A$}
  		\DisplayProof
\end{tabular}
\end{center}
since all the rules in $\mathsf{L}-\mathbf{EF}$ are standard, we have $A_1, \ldots\, A_l \vdash_{\mathsf{L}} A$. By Remark \ref{RemLM}, $A_1^t, \ldots\, A_l^t \vdash_{\mathsf{M}} A^t$. Hence, all the new rules in $\mathbf{P}$ are standard with respect to $\mathsf{M}$. \\
Now, let $\pi= \phi_1, \ldots, \phi_m=\phi$ be a proof for $\phi$ in $\mathsf{L}-\mathbf{EF}$. Then, each $\phi_i$ is either an extension axiom, or it is derived from $\{\phi_{j_1}, \ldots, \phi_{j_l}\}$ such that all $j_r$'s are less than $i$. It is clear that $\pi^t= \phi_1^t, \ldots, \phi_m^t=\phi^t$ is a proof in $\mathbf{P}$. Because the translation $t$ of the extension axiom of $\mathsf{L}-\mathbf{EF}$ will be the extension axiom of $\mathsf{M}-\mathbf{EF}$ and moreover,
\begin{center}
  	\begin{tabular}{c}
  		\AxiomC{$\phi_{j_1}^t$}
  		\AxiomC{$\ldots$}
	  	\AxiomC{$\phi_{j_l}^t$}
	  	\TrinaryInfC{$\phi_i^t$}
  		\DisplayProof
\end{tabular}
\end{center}
is an instance of a rule in $\mathbf{P}$. Furthermore, the number of proof-lines stay the same, i.e., $\lambda(\pi) = \lambda (\pi^t)$. \\
Note that the above construction also works for the case of Frege systems. It is easy to see that the translation of every proof in $\mathsf{L}-\mathbf{F}$ will be a proof in $\mathsf{M}-\mathbf{F}$, and the number of proof-lines stay the same. Hence, for a proof $\pi$ of $\Theta^*_n$ in $\mathsf{L}-\mathbf{F}$, $\pi^t$ will be a proof of $(\Theta^*_n)^t$ in $\mathsf{M}-\mathbf{F}$. Moreover, by the discussion in the proof of Theorem \ref{MainTheorem}, we can find a proof for $(\Theta^*_n)^t \to \Theta_n$ whose number of proof-lines is polynomial in $n$. Therefore, the bound on the number of proof-lines follows, again by Theorem \ref{JerabekAsli}.\\

For part $(ii)$, the strategy is similar to that of part $(i)$, and again we use Theorem \ref{MainTheorem}. First, note that any $\mathsf{L}-\mathbf{EF}$ is at least as strong as $\mathbf{BPC}$. Second, we provide an extended Frege system $\mathbf{P}$ for $\mathsf{M}$ such that $\mathsf{L}-\mathbf{EF} \leq \mathbf{P}$. First, fix an extended Frege system $\mathbf{Q}$ for the logic $\mathsf{M}$. Add the rules of the $\mathsf{L}-\mathbf{EF}$ system wrt to $\mathsf{M}$ to $\mathbf{Q}$. The resulting system, which we denote by $\mathbf{P}$, is an extended Frege system for the logic $\mathsf{M}$. The reason is similar to the argument in the part $(i)$, using the facts that $\mathsf{L} \subseteq \mathsf{M}$ and all the rules of $\mathsf{L}-\mathbf{EF}$ system wrt $\mathsf{M}$ are $\mathsf{M}$-standard. Therefore, the lower bound on the length of the proof follows. For Frege systems and the lower bound on the number of proof-lines, we use an argument similar to the one in part $(i)$.
\end{proof}

The following corollary is the main concrete application of the paper.

\begin{cor} \label{CorKolli}
\begin{itemize}
\item
Let $S$ be any subset of $\{e,c,i,o\}$, and $\mathsf{L}$ be $\mathsf{FL_S}$, or any of the logics of the sequent calculi in Table \ref{table:pekh}. Then, %$\mathsf{L} \vdash \Theta^*_n$ and 
the number of lines of every proof of $\Theta^*_n$ in $\mathsf{L}-\mathbf{F}$ is exponential in $n$ and every proof of $\Theta^*_n$ in $\mathsf{L}-\mathbf{EF}$ has length exponential in $n$.
\item
Let $\mathsf{L}$ be $\mathsf{BPC}$ or $\mathsf{EBPC}$ and $\mathsf{M}$ be a super-intuitionistic logic of infinite branching with a Frege system (e.g., $\mathsf{M}= \mathsf{IPC}$). Then, %$\mathsf{L} \vdash (\Theta^*_n)^t$ and 
the number of lines of every proof of $(\Theta^*_n)^t$ in any $\mathsf{L}-\mathbf{F}$ system wrt $\mathsf{M}$ is exponential in $n$ and every proof of $\Theta^*_n$ in any $\mathsf{L}-\mathbf{EF}$ system wrt $\mathsf{M}$ has length exponential in $n$.
\end{itemize}
\end{cor}

\section{The lower bound for sequent calculi} \label{SectionSequentCalculus}
So far, we have provided a lower bound for proof systems for logics as least as strong as $\mathsf{FL}$ and polynomially simulated by an extended Frege system for a super-intuitionistic logic of infinite branching. It is very desirable to see if the lower bound also applies to proof systems for logics outside this range, for instance their classical counterparts. The result in this section is an attempt in this direction and we reach a positive answer for any proof system polynomially weaker than $\mathbf{CFL^-_{ew}}$, which is the system $\mathbf{CFL_{ew}}$ without the cut rule. For that matter, we first transfer the lower bound from the previous section to the sequent-style proof system $\mathbf{FL_S}$ for any $S \subseteq \{e, c, i, o\}$. Then we use the observation that any cut-free proof of a single-conclusion sequent in the $0$-free fragment of $\mathbf{CFL_{ew}}$ is also an $\mathbf{FL_{ew}}$-proof.

\begin{dfn} 
In any of the sequent calculi defined so far, a \emph{line} in a proof is a sequent of the form $\Gamma \Rightarrow \Delta$. We denote the number of proof-lines in a proof $\pi$ in a sequent calculus by $\mathbf{\lambda}(\pi)$. It is obvious that the number of proof-lines of a sequent is less than or equal to the length of the proof, i.e., the number of symbols in the proof. 
\end{dfn}

\begin{thm} \label{ThmSequentCalculusFLToFrege}
Let $S$ be a subset of $\{e,c,i,o\}$. There exists a Frege system $\mathbf{P}$ for the sequent calculus $\mathbf{FL_S}$ such that for any sequence of formulas $\Gamma=\gamma_1, \ldots, \gamma_m$ and any formula $A$, if $\mathbf{FL_S} \vdash^{\pi} \Gamma \Rightarrow A$, then 
\[
\mathbf{P} \vdash^{\pi'} \bigast_{i=1}^m \gamma_i \setminus A
\]
and $\lambda (\pi') = \lambda (\pi)$.
\end{thm}

\begin{proof}
The proof is similar to the proof of Theorem \ref{ThmFregeForBPC}. As noted in the discussion after Definition \ref{DfnSubstructuralLogic} (when $\Sigma$ is empty), since $\mathbf{FL_S} \vdash \Gamma \Rightarrow A$, we have $\Gamma \vdash_{\mathsf{FL_S}} A$. Therefore, for any Frege system $\mathbf{Q}$ for the logic $\mathsf{FL_S}$, by strong completeness in Definition \ref{DfnFrege}, we have $\Gamma \vdash_{\mathbf{Q}} A$. Fix such $\mathbf{Q}$. The method is developing a Frege system $\mathbf{P}$ for $\mathsf{FL_S}$ by transforming all the axioms and rules of the sequent calculus $\mathbf{FL_S}$ to Frege rules in the new system. For the sake of completeness, we also add $\mathbf{Q}$ to the resulting system.\\
Recall that for $\Gamma=\emptyset$, the formula $\bigast \Gamma$ is defined as $1$ and for any single-conclusion sequent $T= (\Gamma \Rightarrow \Delta)$ by $I(T)$, i.e., the interpretation of the sequent $T$, we mean $\bigast \Gamma \setminus \Delta$, if $\Delta$ is non-empty, and $\bigast \Gamma \setminus 0$ when $\Delta=\emptyset$. Now, define $\mathbf{P}$ as the system consisting of the rules of $\mathbf{Q}$ plus the following rules: for the axiom $T$ in the sequent calculus $\mathbf{FL_S}$ add
\begin{center}
  	\begin{tabular}{c}
  		\AxiomC{}
	  	\UnaryInfC{$I(T)$}
  		\DisplayProof
\end{tabular}
\end{center}
and for any rule in the sequent calculus $\mathbf{FL_S}$ of the form
\begin{center}
  	\begin{tabular}{c}
  		\AxiomC{$T_1$}
  		\AxiomC{$\ldots$}
  		\AxiomC{$T_m$}
	  	\TrinaryInfC{$T$}
  		\DisplayProof
\end{tabular}
\end{center}
add the following rule
\begin{center}
  	\begin{tabular}{c}
  		\AxiomC{$I(T_1)$}
  		\AxiomC{$\ldots$}
  		\AxiomC{$I(T_m)$}
	  	\TrinaryInfC{$I(T)$}
  		\DisplayProof
\end{tabular}
\end{center}
where $m=1$ or $m=2$. We have to show that $\mathbf{P}$ is a Frege system for the logic $\mathsf{FL_S}$. First, since $\mathbf{Q}$ is strongly complete wrt $\mathsf{FL_S}$, then so is $\mathbf{P}$. Now, we have to show that the new rules are standard wrt $\mathsf{FL_S}$, i.e., for any rule of the form
\begin{center}
  	\begin{tabular}{c}
  		\AxiomC{$I(T_1)$}
  		\AxiomC{$\ldots$}
  		\AxiomC{$I(T_m)$}
	  	\TrinaryInfC{$I(T)$}
  		\DisplayProof
\end{tabular}
\end{center}
in $\mathbf{P}$ we have to show that $I(T_1), \ldots , I(T_m) \vdash_{\mathsf{FL_S}}I(T)$. For that matter, note that since the sequent calculus $\mathbf{FL_S}$ has the cut rule, we have $\Rightarrow I(T_i) \vdash_{\mathbf{FL_S}} T_i$ using
\[
\mathbf{FL_S} \vdash \Gamma, \bigast \Gamma \setminus A \Rightarrow A.
\]
Now, use the corresponding rule, $T_1, \ldots, T_m \vdash_{\mathbf{FL_S}} T$, the fact that $T \vdash_{\mathbf{FL_S}} \Rightarrow I(T)$, and the cut rule to show  $\Rightarrow I(T_1), \ldots, \Rightarrow I(T_m) \vdash_{\mathbf{FL_S}} \Rightarrow I(T)$. Therefore, again by the discussion after Definition \ref{DfnSubstructuralLogic} (where $\Sigma= I(T_1), \ldots , I(T_m)$ and $\Gamma$ is empty), we have $I(T_1), \ldots , I(T_m) \vdash_{\mathsf{FL_S}} I(T)$. Hence, $\mathbf{P}$ is a Frege system for $\mathsf{FL_S}$. \\
For the number of proof-lines, note that if $\pi= T_1, \ldots, T_n$ is a proof for $T_n=(\Gamma \Rightarrow A)$ in $\mathbf{FL_S}$, then it is easy to see that $I(T_1), \ldots, I(T_n)$ will be a proof for $\bigast_{i=1}^m \gamma_i \setminus A$ in $\mathbf{P}$. Therefore, $\lambda(\pi') = \lambda(\pi)$.
\end{proof}

\begin{cor}
For any $S \subseteq \{e, i, o, c\}$, we have $\mathbf{FL_S} \vdash \, \Rightarrow \Theta^*_n$ and the number of lines of any proof of this sequent is exponential in $n$.
\end{cor}

By a $0$-free formula in $\mathsf{CFL_{ew}}$, we mean a formula only consisting of propositional variables, the constant $1$, and the connectives $\{\wedge, \vee, \to, *\}$.

\begin{lem} \label{EmptyConclusionNotProvable}
If $\Gamma$ is a sequence of $0$-free formulas, then $\mathbf{CFL_{ew}^-} \nvdash \Gamma \Rightarrow $.
\end{lem}

\begin{proof}
Suppose $(\Gamma \Rightarrow)$ has a proof in $\mathbf{CFL_{ew}^-}$. Since the proof is cut-free and $\Gamma$ is $0$-free, by the subformula property of $\mathbf{CFL_{ew}^-}$, the whole proof is also $0$-free. Therefore, there is no axiom in the proof with an empty succedent, because such an axiom must be in the form $(0 \Rightarrow)$, which is not $0$-free. Moreover, if the succedent of the conclusion of any rule is empty, then the succedent of at least one of its premises must be empty, as well. The reason is the following. First, note that the last rule is not an axiom, as stated. It cannot be a right rule either, because they always have at least one formula in the succedent of their conclusion. And for the left rules, the claim is evident by a simple case checking. The only non-trivial case to check is $(L \to)$ which also has such a premise:
\begin{center}
  	\begin{tabular}{c}
  			  	\AxiomC{$\Phi \Rightarrow \phi  $}
	  	\AxiomC{$\Pi, \psi, \Sigma \Rightarrow $}
	  	\RightLabel{($L \to$)}
	  	\BinaryInfC{$\Pi, \Phi, \phi \to \psi , \Sigma \Rightarrow$}
  		\DisplayProof
\end{tabular}
\end{center}
Therefore, any sequent in the proof with an empty succedent has also a premise with an empty succedent. This is clearly a contradiction.
\end{proof}

The following theorem, which is of independent interest, states that for $0$-free formulas, a cut-free proof for a single-conclusion sequent in $\mathbf{CFL_{ew}}$ is also a proof for the same sequent in $\mathbf{FL_{ew}}$. 

\begin{thm} \label{PositiveFormulasCLFe}
Suppose $\Gamma$ is a sequence of $0$-free formulas and $A$ is a $0$-free formula. Then any proof $\pi$ for $\Gamma \Rightarrow A$ in $\mathbf{CFL_{ew}^-}$ is also a proof in $\mathbf{FL_{ew}}$.
\end{thm}

\begin{proof}
The sketch of the proof is the following: suppose $\pi$ is a cut-free proof in $\mathbf{CFL_{ew}}$. Since $\Gamma \Rightarrow A$ is $0$-free and the proof is cut-free, then all the formulas in the proof must be $0$-free. Then, throughout the proof the number of formulas in the succedent of the sequents does not decrease. The reason lies in the fact that neither the cut rule nor the contraction rules are present. Hence, in the special case that the sequent is also single-conclusion, the succedents of all the sequents in the whole proof will contain exactly one formula. Therefore, the proof is in $\mathbf{FL_{ew}}$.\\
Now, we state the proof extensively. Let $\pi$ be a proof for $\Gamma \Rightarrow A$ in $\mathbf{CFL_{ew}^-}$. By induction on the structure of $\pi$ we will show it is also a proof for the same sequent in $\mathbf{FL_{ew}}$. As stated in the proof of Lemma \ref{EmptyConclusionNotProvable}, every formula in the proof must be $0$-free.\\
If $\Gamma \Rightarrow A$ is an instance of an axiom in $\mathbf{CFL_{ew}^-}$, then it is either $\Rightarrow 1$ or an instance of the axiom $\phi \Rightarrow \phi$, which are both also axioms in the sequent calculus $\mathbf{FL_e}$.
For the induction step, note that the last rule in the proof cannot be $(0 w)$. For all the other rules (except for the rule $(L \to)$), it is easy to see that since the conclusion of the rule is single-conclusion, then every premise must also be single-conclusion. 
It remains to investigate the case where the last rule used in the proof is $(L \to)$:

\begin{center}
 \begin{tabular}{c}
\AxiomC{$\pi_1$} 
	  	 \noLine
 \UnaryInfC{$\Phi \Rightarrow \phi, \Lambda$}
\AxiomC{$\pi_2$} 
	  	 \noLine
 \UnaryInfC{$\Pi, \psi, \Sigma \Rightarrow \Delta$}
 \RightLabel{$(L \to)$} 
 \BinaryInfC{$\Pi, \Phi, \phi \to \psi, \Sigma \Rightarrow \Delta, \Lambda$}
 \DisplayProof
\end{tabular}
\end{center}
There are two possibilities; either $\Lambda$ is empty and $\Delta$ is equal to $A$

\begin{center}
 \begin{tabular}{c}
\AxiomC{$\pi_1$} 
	  	 \noLine
 \UnaryInfC{$\Phi \Rightarrow \phi$}
\AxiomC{$\pi_2$} 
	  	 \noLine
 \UnaryInfC{$\Pi, \psi, \Sigma \Rightarrow A$}
 \RightLabel{$(L \to)$} 
 \BinaryInfC{$\Pi, \Phi, \phi \to \psi, \Sigma \Rightarrow A$}
 \DisplayProof
\end{tabular}
\end{center}
or $\Delta$ is empty and $\Lambda$ is equal to $A$

\begin{center}
 \begin{tabular}{c}
\AxiomC{$\pi_1$} 
	  	 \noLine
 \UnaryInfC{$\Phi \Rightarrow \phi, A$}
\AxiomC{$\pi_2$} 
	  	 \noLine
 \UnaryInfC{$\Pi, \psi, \Sigma \Rightarrow $}
 \RightLabel{$(L \to)$} 
 \BinaryInfC{$\Pi, \Phi, \phi \to \psi, \Sigma \Rightarrow A$}
 \DisplayProof
\end{tabular}
\end{center}
In the former since both premises are single-conclusion, by induction hypothesis, $\pi_1$ and $\pi_2$ are proofs in $\mathbf{FL_{ew}}$ and by applying the rule $(L \to)$ we obtain a proof for $\Gamma \Rightarrow A$. On the other hand, the latter cannot happen since the right premise is of the form $\Pi, \psi, \Sigma \Rightarrow$ and the antecedent of this sequent is $0$-free. Therefore, Lemma \ref{EmptyConclusionNotProvable} implies that it is not provable in $\mathbf{CFL_{ew}^-}$.
\end{proof}

\begin{thm} \label{ThmLowerBoundCFLe-}
The formulas
\[
\tilde{\Theta}^{*}_{n,k}:=  [\bigast_{i,j} ((p_{i,j} \wedge 1) \vee (q_{i,j} \wedge 1))] \;\; \to
\]
\[
[\bigast_{i,l} ((s_{i,l} \wedge 1) \vee (s'_{i,l} \wedge 1)) \to \alpha^k_n (\bar{p}, \bar{s}, \bar{s'})]\; \vee \; [\bigast_{i,l}((r_{i,l} \wedge 1) \vee  (r'_{i,l} \wedge 1)) \to \beta^{k+1}_n (\bar{q}, \bar{r}, \bar{r'})].
\]
are provable in $\mathbf{CFL_{e}^-}$. Moreover, defining $\tilde{\Theta}^{*}_{n}=\tilde{\Theta}^{*}_{n,\lfloor \sqrt{n} \rfloor}$, every $\mathbf{CFL_{ew}^-}$-proof of $\tilde{\Theta}^{*}_n$ contains at least $2^{\Omega(n^{1/4})}$ proof-lines and hence has length exponential in terms of the length of $\tilde{\Theta}^{*}_n$.
\end{thm}

\begin{proof}
Since formulas $\Theta^{*}_n$ are provable in $\mathbf{FL}$, by Theorem \ref{HardTautologiesFLe}, they are also provable in $\mathbf{FL_{e}}$, and hence in $\mathbf{CFL_{e}}$. However, since in $\mathbf{FL_{e}}$ and $\mathbf{CFL_{e}}$ the exchange rules are present, as stated in Preliminaries the connectives $\setminus$ and $/$ can be substituted by $\to$, see Remark \ref{RemImplication}. Therefore, the tautologies $\Theta^*_n$ will have the more recognizable form $\tilde{\Theta}^{*}_n$.
Using the cut elimination theorem for $\mathbf{CFL_{e}}$, formulas $\tilde{\Theta}^{*}_n$ are also provable in $\mathbf{CFL_{e}^-}$ and therefore also in $\mathbf{CFL_{ew}^-}$. By Theorem \ref{PositiveFormulasCLFe}, since $\tilde{\Theta}^{*}_n$ are $0$-free formulas, any cut-free proof for these formulas in $\mathbf{CFL_{ew}^-}$ is also a proof in $\mathbf{FL_{ew}}$. However, Theorem \ref{ThmSequentCalculusFLToFrege} guaranties that these proofs contain at least $2^{\Omega(n^{1/4})}$ proof-lines and hence the lengths of these proofs are exponential in terms of the length of $\tilde{\Theta}^{*}_n$. 
\end{proof}

\begin{rem}
So far, we do not have any method to extend the lower bound to the calculus $\mathbf{CFL_e}$, where the cut rule is present. Note that since there are no non-trivial lower bounds for the sequent calculus $\mathbf{LK}$, we can not use a similar argument as in the proof of Theorem \ref{HardTautologiesFLe}. 
\end{rem}

\begin{cor} 
For any proof system $\mathbf{P}$ such that $\mathbf{P}$ is at least as strong as $\mathbf{FL_e}$ and $\mathbf{P} \leq \mathbf{CFL_{ew}^-}$, there is an exponential lower bound on the length of proofs in $\mathbf{P}$. As a result, there are exponential lower bounds on the length of proofs in sequent calculi $\mathbf{CFL_{e}^-}$, $\mathbf{CFL_{ei}^-}$, and $\mathbf{CFL_{eo}^-}$.
\end{cor}

\begin{proof}
It follows from Theorem \ref{ThmLowerBoundCFLe-}.
\end{proof}
In the end, we will provide a result similar to Theorem \ref{ThmSequentCalculusFLToFrege} for the sequent calculus $\mathbf{BPC}$. 

\begin{cor}
We have $\mathbf{BPC} \vdash \, \Rightarrow (\Theta^*_n)^t$ and the number of lines of any proof of this sequent is exponential in $n$.
\end{cor}
\begin{proof}
Take the Frege system $\mathbf{P}$ for $\mathsf{BPC}$ wrt $\mathsf{IPC}$ introduced in Theorem \ref{ThmFregeForBPC}. Using Theorem \ref{ThmFregeForBPC}, since $\mathbf{BPC} \vdash \, \Rightarrow (\Theta^*_n)^t$, we have $\mathbf{P} \vdash \top \to (\Theta^*_n)^t$. By the modus ponens rule, which is available in $\mathbf{P}$, we get $\mathbf{P} \vdash (\Theta^*_n)^t$. Using the second part of Theorem \ref{ThmFLBPC-EF}, we get the lower bound.
\end{proof}

\textbf{Acknowledgment} My sincere thanks go to Pavel Pudl\'{a}k for his careful reading of this work and his invaluable comments. Many thanks are also due to Emil Je\v{r}\'{a}bek for bringing the main question of the paper to attention and for fruitful discussions. I am also very thankful to Amir Akbar Tabatabai for discussions and useful suggestions. Many thanks also go to Hiroakira Ono, Mohammad Ardeshir and Majid Alizadeh for interesting discussions. I am very grateful to Rosalie Iemhoff and thankful for the hospitality of the Department of Philosophy of Utrecht University where part of this research was done while I was visiting there.

\end{document}